\documentclass[runningheads]{llncs}

\usepackage{amsmath}
\usepackage{amsfonts}
\usepackage{mathtools}
\usepackage{thm-restate}
\usepackage{proof}
\usepackage{stmaryrd}

\usepackage[final]{microtype}

\usepackage[backend=biber,style=numeric-comp,mincrossrefs=10,sortcites=true]{biblatex}
\usepackage{graphicx}

\usepackage{hyperref}
\usepackage[dvipsnames]{xcolor}

\usepackage{cleveref}

\crefname{diagram}{diagram}{diagrams}
\crefname{intn}{interpretation}{interpretations}
\crefname{property}{property}{properties} %
\crefname{proposition}{proposition}{propositions}
\crefname{lemma}{lemma}{lemmas}
\crefname{theorem}{theorem}{theorems}

\usepackage[all,cmtip,2cell]{xy}
\usepackage{tikz-cd}

\newcommand{\ms}{\mathsf}
\newcommand{\mb}{\mathbf}

\newcommand{\rn}[1]{(\textsc{#1})}

\newcommand{\FIX}{\mathsf{FIX}}
\newcommand{\GFIX}{\mathsf{GFIX}}
\newcommand{\UNF}{\mathsf{UNF}}
\newcommand{\sfix}[1]{\ensuremath{#1^\dag}}

\newcommand{\sembr}[1]{\left\llbracket#1\right\rrbracket}
\newcommand{\ppi}[1]{\llparenthesis\, #1 \,\rrparenthesis}
\newcommand{\genbr}[1]{\{\!| #1 |\!\}}

\newcommand{\C}{\mathcal{C}}
\newcommand{\N}{\mathbb{N}}
\renewcommand{\O}{\ensuremath{\mathbf{O}}}
\newcommand{\CFP}{\ensuremath{\mathbf{CFP}}}

\newcommand{\strc}[1]{{#1}_{\bot!}}

\newcommand{\lcto}{\xrightarrow{\ms{l.c.}}}

\newcommand{\nto}{\Rightarrow}

\newcommand{\cel}[1]{\int #1} %

\newcommand{\algc}[2]{{#1}^{#2}}
\newcommand{\coalgc}[2]{{#1}_{#2}}

\newcommand{\lk}[3]{(#1, #2, #3)}%

\newcommand{\Trec}[2]{\rho #1.#2}

\DeclareMathOperator{\Cone}{Cone}
\DeclareMathOperator{\colim}{colim}

\usepackage{stackengine,scalerel}
\newcommand\operatorupX[1]{\,\ThisStyle{\ensurestackMath{%
      #1\stackengine{-0pt}{\,}{\SavedStyle\!^{\mathord{\uparrow}}}{O}{l}{F}{T}{S}}}\!}
\makeatletter
\newcommand\operatorup[1]{\!\mathop{\operatorupX{#1}}\@ifnextchar\{{\,}{\@ifnextchar[{\,}{\@ifnextchar({\,}{}}}}
\makeatother
\newcommand{\dirsqcup}{\operatorup{\bigsqcup}}

\newcommand{\dirsup}{\dirsqcup}

\newcommand{\defin}[1]{{\usefont{T1}{lmss}{b}{n}{#1}}}

\usepackage{xspace}
\newcommand{\ie}{i.e.\@\xspace}
\newcommand{\eg}{e.g.\@\xspace}
\newcommand{\cf}{cf.\@\xspace}
\makeatletter
\newcommand*{\etc}{%
  \@ifnextchar{.}%
  {etc}%
  {etc.\@\xspace}%
}
\makeatother
\newcommand{\Links}{\ensuremath{\mathbf{Links}}\xspace}

\usepackage[draft]{fixme}
\fxsetup{theme=color}
\FXRegisterAuthor{rk}{anrk}{RAK}

\usepackage{environ}

\newcommand{\repeattheorem}[1]{%
  \begingroup
  \renewcommand{\thetheorem}{\ref{#1}}%
  \expandafter\expandafter\expandafter\theorem
  \csname reptheorem@#1\endcsname
  \endtheorem
  \endgroup
}

\NewEnviron{rtheorem}[1]{%
  \global\expandafter\xdef\csname reptheorem@#1\endcsname{%
    \unexpanded\expandafter{\BODY}%
  }%
  \expandafter\theorem\BODY\unskip\label{#1}\endtheorem%
}

\let\endllncsproposition\endproposition%

\newcommand{\repeatproposition}[1]{%
  \begingroup%
  \renewcommand{\theproposition}{\ref{#1}}%
  \expandafter\expandafter\expandafter\proposition%
  \csname repproposition@#1\endcsname%
  \endproposition%
  \endgroup%
}

\NewEnviron{rproposition}[1]{%
  \global\expandafter\xdef\csname repproposition@#1\endcsname{%
    \unexpanded\expandafter{\BODY}%
  }%
  \expandafter\proposition\BODY\unskip\label{#1}\endllncsproposition%
}

\newcommand{\repeatlemma}[1]{%
  \begingroup%
  \renewcommand{\thelemma}{\ref{#1}}%
  \expandafter\expandafter\expandafter\lemma%
  \csname replemma@#1\endcsname%
  \endlemma%
  \endgroup%
}

\NewEnviron{rlemma}[1]{%
  \global\expandafter\xdef\csname replemma@#1\endcsname{%
    \unexpanded\expandafter{\BODY}%
  }%
  \expandafter\lemma\BODY\unskip\label{#1}\endllncslemma%
}

\allowdisplaybreaks

\begin{document}
\title{Parametrized Fixed Points on \O-Categories and Applications to Session Types}
\titlerunning{Parametrized Fixed Points on \O-Categories}
\author{Ryan Kavanagh\orcidID{0000-0001-9497-4276}}
\authorrunning{R. Kavanagh}
\institute{Computer Science Department, Carnegie Mellon University, Pittsburgh, PA 15213-3890, USA \email{rkavanagh@cs.cmu.edu} \url{https://rak.ac/}}
\maketitle              %
\begin{abstract}
  \O-categories~\cite{wand_1977:_fixed_point_const,smyth_plotkin_1982:_categ_theor_solut} generalize categories of domains to provide just the structure required to compute fixed points of locally continuous functors.
  Parametrized fixed points are of particular interest to denotational semantics and are often given by ``dagger operations''~\cite{bloom_esik_1993:_iterat_theor,bloom_esik_1996:_fixed_point_operat,bloom_esik_1995:_some_equat_laws_initial}.
  We generalize existing techniques to define a functorial dagger operation on locally continuous functors between O-categories.
  We show that this dagger operation satisfies the Conway identities~\cite{bloom_esik_1996:_fixed_point_operat}, a collection of identities used to axiomatize iteration theories.
  We study the behaviour of this dagger operation on natural transformations and consider applications to semantics of session-typed languages.

\end{abstract}

\section{Introduction}
\label{sec:introduction}

Recursive types are ubiquitous in functional languages.
For example, in Standard ML we can define the type of (unary) natural numbers as:
\begin{verbatim}
datatype nat = Zero | Succ of nat
\end{verbatim}
This declaration specifies that a \verb!nat! is either zero or the successor of some natural number.
Semantically, we can think of \verb!nat! as a domain $D$ satisfying the domain equation $D \cong (\mathtt{Zero} : \{\bot\}) \uplus (\mathtt{Succ} : D)$, where $\uplus$ forms the labelled disjoint union of domains.
Equivalently, we can think of $D$ as a fixed point of the functor $F_{\mathtt{nat}}(X) = (\mathtt{Zero} : \{\bot\}) \uplus (\mathtt{Succ} : X)$ on a category of domains.

Mutually-recursive data types give rise to a similar interpretation.
Consider, for example, the types of even and odd natural numbers:
\begin{verbatim}
datatype even = Zero | E of odd
     and odd  = O of even
\end{verbatim}
This declaration specifies that an even number is either zero or the successor of an odd number, and that an odd number is the successor of an even one.
The types $\mathtt{even}$ and $\mathtt{odd}$ respectively denote solutions $D_{\mathtt{e}}$ and $D_{\mathtt{o}}$ to the system of domain equations:
\[
  D_{\mathtt{e}} \cong (\mathtt{Zero} : \{\bot\}) \uplus (\mathtt{E} : D_{\mathtt{o}}) \quad \text{and} \quad D_{\mathtt{o}} \cong (\mathtt{O} : D_{\mathtt{e}}).
\]
These are solutions to the system of equations:
\begin{align}
  X_{\mathtt{e}} &\cong F_{\mathtt{even}}(X_{\mathtt{e}}, X_{\mathtt{o}})\label{eq:proposal:3}\\
  X_{\mathtt{o}} &\cong F_{\mathtt{odd}}(X_{\mathtt{e}}, X_{\mathtt{o}})\label{eq:proposal:4}
\end{align}
where $F_{\mathtt{even}}$ and $F_{\mathtt{odd}}$ are the functors $F_{\mathtt{even}}(X_{\mathtt{e}}, X_{\mathtt{o}}) = (\mathtt{Zero} : \{\bot\}) \uplus (\mathtt{E} : X_{\mathtt{o}})$ and $F_{\mathtt{odd}}(X_{\mathtt{e}}, X_{\mathtt{o}}) = (\mathtt{O} : X_{\mathtt{e}})$.
We can use Beki{\v{c}}'s rule~\cite[\S~2]{bekic_1984:_defin_operat_gener} to solve this system of equations.
To do so, we think of \cref{eq:proposal:3} as a family of equations parametrized by $X_{\mathtt{o}}$.
If we could solve for $X_{\mathtt{e}}$, then we would get a parametrized family of solutions $\sfix{F_{\mathtt{even}}}(X_{\mathtt{o}})$ such that:
\begin{equation}
  \label{eq:proposal:5}
  \sfix{F_{\mathtt{even}}}(X_{\mathtt{o}}) \cong F_{\mathtt{even}}(\sfix{F_{\mathtt{even}}}(X_{\mathtt{o}}), X_{\mathtt{o}})
\end{equation}
for all domains $X_{\mathtt{o}}$.
Substituting this for $X_{\mathtt{e}}$ in \cref{eq:proposal:4} gives the domain equation
\[
  X_{\mathtt{o}} \cong F_{\mathtt{odd}}(\sfix{F_{\mathtt{even}}}(X_{\mathtt{o}}), X_{\mathtt{o}}).
\]
Solving for $X_{\mathtt{o}}$ gives the solution $D_{\mathtt{o}}$.
Substituting $D_{\mathtt{o}}$ for $X_{\mathtt{o}}$ in \cref{eq:proposal:5}, we see that $D_{\mathtt{e}} = \sfix{F_{\mathtt{even}}}(D_{\mathtt{o}})$ is the other part of the solution.

The above example motivates techniques for solving parametrized domain equations.
These are well understood.
For example, given a suitable functor $F : \mb{D} \times \mb{E} \to \mb{E}$ on suitable categories of domains, \cite[Proposition~5.2.7]{abramsky_jung_1995:_domain_theor} gives a recipe for constructing a functor $\sfix{F} : \mb{D} \to \mb{E}$ such that for all objects $D$ of $\mb{D}$, $\sfix{F}D \cong F(D, \sfix{F}D)$.
Is the mapping $F \mapsto \sfix{F}$ functorial?
Semantically, substitution is typically interpreted as composition~\cite[\S~3.4]{crole_1993:_categ_types}.
If the interpretations of recursive types are to respect substitution, then the mapping $F \mapsto \sfix{F}$ must be natural in $\mb{D}$.
Is it?
What other properties does it satisfy?

Families of parametric fixed points arise elsewhere in mathematics.
An external \defin{dagger operation}~\cites[Definition~2.6]{bloom_esik_1995:_some_equat_laws_initial}[7]{bloom_esik_1996:_fixed_point_operat} on a cartesian closed category $\mb{C}$ is a family $\dagger_{A,B} : \mb{C}(A \times B, B) \to \mb{C}(A, B)$ of set theoretic functions for each pair of objects $A$ and $B$ in $\mb{C}$.
Of particular interest are dagger operations that satisfy the (cartesian) Conway identities.
These identities imply many identities~\cite[\S~3.3]{bloom_esik_1996:_fixed_point_operat} useful for semantic reasoning, such as Beki{\v{c}}'s rule.
They also axiomatize a decidable theory~\cite{bernatsky_esik_1998:_seman_flowc_progr}, and dagger operations that satisfy them are closely related to the trace operator~\cites{joyal_1996:_traced_monoid_categ}[Theorem~7.1]{hasegawa_1999:_recur_cyclic_sharin}[281]{benton_hyland_2003:_traced_premon_categ}.
Does the above dagger operation satisfy the Conway identities?

\O-categories~\cite{smyth_plotkin_1982:_categ_theor_solut} generalize categories of domains to provide just the structure required to compute fixed points of functors.
In this paper, we present a locally continuous dagger operation acting on locally continuous functors between \O-categories (\cref{sec:funct-fixed-points}).
In \cref{sec:canon-fixed-points}, we show that these parametrized families of solutions are canonical, \ie, that they determine initial algebras and terminal coalgebras.
We show that the dagger operation satisfies the Conway identities and the power identities up to isomorphism in \cref{sec:conway-identities}.
As an application, in \cref{sec:appl-semant} we see that properties of our dagger operation are essential for defining and reasoning about semantics of session-typed languages with recursion.

\section{Background and Notation}
\label{sec:background-notation}

We review some standard definitions and fix our notation.
For more details, see \cites{abramsky_jung_1995:_domain_theor}[Chapter~10]{gunter_1992:_seman_progr_languag}{smyth_plotkin_1982:_categ_theor_solut}.

We write $\dirsup A$ for the \defin{directed supremum} of a directed set $A$.

We use the arrow $\to$ when describing single morphisms, and $\nto$ for natural transformations or families of morphisms.
In 2-categories, we use $\to$ and $\ast$ for horizontal morphisms and composition, and $\nto$ and $\circ$ for vertical morphisms and composition.
Given two objects $A$ and $B$ of a category $\mb{C}$, we write $\mb{C}(A, B)$ for the \defin{homset} of morphisms $A \to B$ when $\mb{C}$ is locally small.
We write $[A \to B]$ for the \defin{internal hom} in $\mb{C}$ if it exists.
We write $\bot_{\mb{C}}$ or just $\bot$ for the \defin{initial object} of a category $\mb{C}$.
We also write $\bot$ for the unique cone $\bot_{\mb{C}} \nto \ms{id}_{\mb{C}}$ witnessing the initiality of $\bot_{\mb{C}}$.
If $\mb{C}$ has a terminal object isomorphic to its initial object, we call the initial object the \defin{zero object} $0_{\mb{C}}$.
$\mb{C}$ has \defin{zero morphisms} if for all objects $A$ and $D$ there exists a fixed morphism $0_{AD} : A \to D$, and if this family of morphisms satisfies $0_{BD} \circ f = 0_{AD} = g \circ 0_{AC}$ for all morphisms $f : A \to B$ and $g : C \to D$.
$\mb{C}$ has zero morphisms whenever it has a zero object: $0_{AB} = A \to 0 \to B$.

Given functors $F,G : \mb{C} \to \mb{D}$ and $H, I : \mb{D} \to \mb{E}$, the \defin{horizontal composition} $\eta \ast \epsilon : HF \nto IG$ of natural transformations $\eta : H \nto I$ and $\epsilon : F \nto G$ is given by the equal natural transformations $I\epsilon \circ \eta F = \eta G \circ H\epsilon$.
Given a morphism $f : K \to L$ in $\mb{C}$, we abuse notation and write $\eta \ast f : FK \to GL$ for the naturality square $FK \xrightarrow{Ff} FL \xrightarrow{\eta_L} GL = FK\xrightarrow{\epsilon_K} GK \xrightarrow{Gf} GK$.

An \defin{\O-category}~\cites[Definition~5]{smyth_plotkin_1982:_categ_theor_solut} is a category $\mb{K}$ where every hom-set $\mb{K}(C, D)$ is a dcpo, and where composition of morphisms is continuous with respect to the partial ordering on morphisms.
A functor $F : \mb{D} \to \mb{E}$ between $\mb{O}$-categories is \defin{locally continuous} if the maps $f \mapsto F(f) : \mb{D}(D_1, D_2) \to \mb{D}(F(D_1), F(D_2))$ are continuous for all objects $D_1, D_2$ of $\mb{D}$.
We write $[\mb{D} \lcto \mb{E}]$ for the category of locally continuous functors $\mb{D} \to \mb{E}$, with natural transformations as morphisms.
Examples of $\mb{O}$-categories include $\mb{DCPO}$ and functor categories $[\mb{C} \to \mb{D}]$ whenever $\mb{D}$ is an $\mb{O}$-category.

Small \O-categories form a 2-cartesian closed category {\O}.
Horizontal morphisms are locally continuous functors and vertical morphisms are natural transformations.
Many interesting \O-categories, \eg, $\mb{DCPO}$, are not small.
We can use a hierarchy of universes~\cite[\S~3]{schubert_1972:_categ} to make categories of interest small.

An \defin{embedding-projection pair} (e-p-pair)~\cite[Definition~6]{smyth_plotkin_1982:_categ_theor_solut} $(e, p)$  is a pair of morphisms $e : D \to E$ and $p : E \to D$ such that $p \circ e = \ms{id}_D$ and $e \circ p \sqsubseteq \ms{id}_E$.
We call $e$ an \defin{embedding} and $p$ a \defin{projection}.
Given an embedding $e : D \to E$, we write $e^p : E \to D$ for its associated projection.
Given a projection $p : E \to D$, we write $p^e : D \to E$ for its associated embedding.
When $\mb{K}$ is an \O-category, we write $\mb{K}^e$ for the subcategory of $\mb{K}$ whose morphisms are embeddings.

An \defin{$\omega$-chain} in $\mb{K}$ is a diagram $J : \omega \to \mb{K}$.
A cocone $\kappa : J \nto A$ in $\mb{K}^e$ is an \defin{\O-colimit}~\cite[Definition~7]{smyth_plotkin_1982:_categ_theor_solut} if $(\kappa_n \circ \kappa_n^p)_n$ is an ascending chain in $\mb{K}(A, A)$ and $\dirsup_{n \in \N} \kappa_n \circ \kappa_n^p = \ms{id}_A$.
$\mb{K}$ is \defin{\O-cocomplete} if every $\omega$-chain in $\mb{K}^e$ has an \O-colimit~in~$\mb{K}$.
Our interest in $\mb{O}$-colimits is due to \cref{prop:33}, which appears as \cite[Propositions~A and~D]{smyth_plotkin_1982:_categ_theor_solut} and as part of the proof of \cite[Proposition~A]{smyth_plotkin_1982:_categ_theor_solut}.
Parts can also be found in the proof of \cite[Theorem~10.4]{gunter_1992:_seman_progr_languag} or specialized to $\mb{DCPO}$ as \cite[Theorem~3.3.7]{abramsky_jung_1995:_domain_theor}.

\begin{restatable}[{\cite[Propositions~A and~D]{smyth_plotkin_1982:_categ_theor_solut}}]{proposition}{propDD}
  \label{prop:33}
  Let $\mb{K}$ be an \O-category, $\Phi$ an $\omega$-chain in $\mb{K}^e$, and $\alpha : \Phi \nto A$ a cocone in $\mb{K}$.
  \begin{enumerate}
  \item If $\beta : \Phi \nto B$ is a cocone in $\mb{K}^e$, then $(\alpha_n \circ \beta_n^p)_{n \in \N}$ is an ascending chain in $\mb{K}(B, A)$ and the morphism $\theta = \dirsup_{n \in \N} \alpha_n \circ \beta_n^p$ is mediating from $\beta$ to $\alpha$.
  \item If $\beta : \Phi \nto B$ is an \O-colimit, then $\theta$ is an embedding.
  \item If $\alpha$ is an \O-colimit, then $\alpha$ is colimiting in both $\mb{K}$ and $\mb{K}^e$.
  \item If $\alpha$ is colimiting in $\mb{K}$, then $\alpha$ lies in $\mb{K}^e$ and is an \O-colimit.
  \end{enumerate}
\end{restatable}

\section{Functoriality of Fixed Points}
\label{sec:funct-fixed-points}

We show that constructing fixed points of locally continuous functors is a functorial operation.
Typically, the ``canonical'' fixed point of a locally continuous functor $F$ on an \O-category is given by the colimit of the $\omega$-chain $\bot \to F\bot \to F^2\bot \to \dotsb$.
Other fixed points can be constructed using a different ``first link'', \ie, by taking the colimit of a chain $K \to FK \to F^2K \to \dotsb$ generated by a link $k : K \to FK$.

Fix an \O-category $\mb{K}$.
Links form a category $\Links_{\mb{K}}$ where
\begin{itemize}
\item objects are triples $(K, k, F)$ called ``links'', where $K$ is an object of $\mb{K}$, $F : \mb{K} \to \mb{K}$ is locally continuous, and $k : K \to FK$ is an embedding;
\item morphisms $\lk{K}{k}{F} \to \lk{L}{l}{G}$ are pairs $(f, \eta)$ where $f : K \to L$ is a morphism of $\mb{K}$, $\eta : F \nto G$ is a natural transformation in $\mb{K}$, and $f$ and $\eta$ satisfy $l \circ f = (\eta \ast f) \circ k : K \to GL$;
\item composition is given component-wise: $(g, \rho) \circ (f, \eta) = (g \circ f, \rho \circ \eta)$.
\end{itemize}
There exists a locally continuous functor $\Omega : \Links_{\mb{K}} \to [ \omega \to \mb{K}^e ]$ that, given a link $\lk{K}{k}{F}$, generates the $\omega$-chain ${\Omega\lk{K}{k}{F} : K \xrightarrow{k} FK \xrightarrow{Fk} F^2K \xrightarrow{F^2k} \dotsb}$.
The action of $\Omega$ on morphisms uses the horizontal iteration of natural transformations.
Consider functors $H, G : \mb{C} \to \mb{C}$ and a natural transformation $\eta : H \nto G$.
We define the family of \defin{horizontal iterates} $\eta^{(i)} : H^i \nto G^i$, $i \in \N$, by recursion on $i$.
When $i = 0$, $H^0 = G^0 = \ms{id}_{\mb{C}}$ and we define $\eta^{(0)}$ to be the identity natural transformation on $\ms{id}_{\mb{C}}$.
Given $\eta^{(i)}$, we set $\eta^{(i+1)} = \eta \ast \eta^{(i)}$.

We define the functor $\Omega : \Links_{\mb{K}} \to [ \omega \to \mb{K} ]$.
The action of $\Omega\lk{K}{k}{F}$ on morphisms $n \to n + k$ is defined by induction on $k$.
\begin{align}
  \Omega\lk{K}{k}{F}(n) &= F^nK\label{eq:main:1}\\
  \Omega\lk{K}{k}{F}(n \to n) &= \ms{id}_{F^{n} K}\label{eq:main:4}\\
  \Omega\lk{K}{k}{F}(n \to n + k + 1) &= F^{n + k}k \circ \Omega\lk{K}{k}{F}(n \to n + k)\label{eq:main:2}\\
  \Omega(f : K \to L, \eta : F \nto G)_n &= \eta^{(n)} \ast f : F^nK \to G^nL\label{eq:main:3}
\end{align}

\Cref{prop:main:1} generalizes the functor $S : [\mb{C} \to \mb{C}] \to [\mb{\omega} \to \mb{C}]$ of \cites[\S~3]{lehmann_smyth_1977:_data_types}[Lemma~2 (p.~118)]{lehmann_smyth_1981:_algeb_specif_data_types} to form chains with an arbitrary initial link in a locally continuous manner.

\begin{restatable}{proposition}{propMainB}
  \label{prop:main:1}
  \Cref{eq:main:1,eq:main:4,eq:main:2,eq:main:3} define a locally continuous functor $\Omega : \Links_{\mb{K}} \to [ \omega \to \mb{K} ]$.
  For all links $\lk{K}{k}{F}$, $\Omega\lk{K}{k}{F} : \omega \to \mb{K}^e$.
  The natural transformation $\Omega(f, \eta)$ lies in $\mb{K}^e$ whenever $f$ and $\eta$ do.
\end{restatable}

\subsection{General Fixed Points}
\label{sec:general-fixed-points}

We can use $\Omega$ to show that taking fixed points of locally continuous functors is itself functorial and locally continuous.
To do so, we use an $\omega$-colimit functor (\cref{prop:main:10}).
Let $[ \omega \to \mb{K}^e \hookrightarrow \mb{K} ]$ be the subcategory of $[ \omega \to \mb{K} ]$ whose objects are functors $\omega \to \mb{K}^e$ and whose morphisms are natural transformations in $\mb{K}$.

\begin{restatable}{proposition}{propMainBA}
  \label{prop:main:10} %
  Let $\mb{K}$ be an \O-cocomplete \O-category.
  A choice of \O-colimit in $\mb{K}$ for each diagram $\omega \to \mb{K}^e \hookrightarrow \mb{K}$ defines the action on objects of a locally continuous functor $\colim_\omega : [ \omega \to \mb{K}^e \hookrightarrow \mb{K} ] \to \mb{K}^e$.
  Where $\phi : \Phi \nto \colim_\omega \Phi$ and $\gamma : \Gamma \nto \colim_\omega \Gamma$ are the chosen \O-colimits in $\mb{K}$, its action on morphisms $\eta : \Phi \nto \Gamma$ is given by
  \[
    \colim_\omega(\eta : \Phi \nto \Gamma) = \dirsup_{n \in \N} \gamma_n \circ \eta_n \circ \phi_n^p.
  \]
\end{restatable}

We assume below that, whenever $\mb{K}$ is \O-cocomplete, a suitable choice of \O-colimits has been made so that the functor $\colim_\omega$ given by \cref{prop:main:10}~exists.

\begin{restatable}{proposition}{propMainI}%
  \label{prop:main:8}
  Let $\mb{K}$ be an \O-cocomplete \O-category.
  There exists a locally continuous functor $\GFIX = {\colim_\omega} \circ \Omega : \Links_{\mb{K}} \to \mb{K}^e$.
  Its action on morphisms $(f, \eta) : \lk{K}{k}{F} \to \lk{L}{l}{G}$ is given by
  \[
    \GFIX(f, \eta) = \dirsup_{n \in \N} \gamma_n \circ \Omega(f, \eta)_n \circ \phi^p_n
  \]
  where $\phi : \Omega\lk{K}{k}{F} \nto \GFIX\lk{K}{k}{F}$ and $\gamma : \Omega\lk{L}{l}{G} \nto \GFIX\lk{L}{l}{G}$ are the chosen \O-colimits in $\mb{K}$.
\end{restatable}

To show that this indeed gives fixed points, we define the ``unfolding'' functor:

\begin{restatable}{proposition}{propMainC}
  \label{prop:main:2}
  Let $\mb{K}$ be an \O-cocomplete \O-category.
  The following defines a locally continuous functor $\UNF : \Links_{\mb{K}} \to \mb{K}^e$.
  \begin{itemize}
  \item On objects: $\UNF\lk{K}{k}{F} = F(\GFIX\lk{K}{k}{F})$,
  \item on morphisms: $\UNF(f, \eta) = \dirsup_{n \in \N} G\gamma_n \circ \Omega(f,\eta)_{n+1} \circ F\phi_n^p  : \UNF\lk{K}{k}{F} \to \UNF\lk{L}{l}{G}$, where $\phi : \Omega\lk{K}{k}{F} \nto \GFIX\lk{K}{k}{F}$ and $\gamma : \Omega\lk{L}{l}{G} \nto \GFIX\lk{L}{l}{G}$ are the chosen \O-colimits in $\mb{K}$.
  \end{itemize}
\end{restatable}

\Cref{prop:main:3} then tells us that $\GFIX$ produces fixed points.
To the best of our knowledge, the naturality result is new.

\begin{proposition}%
  \label{prop:main:3}
  Let $\mb{K}$ be an \O-cocomplete \O-category.
  There exists a natural isomorphism $\ms{fold} : \UNF \nto \GFIX$ with inverse $\ms{unfold} : \GFIX \nto \UNF$.
  They are explicitly given as follows.
  Let $\lk{K}{k}{F}$ be an object of $\Links_{\mb{K}}$ and let $\kappa : \Omega\lk{K}{k}{F} \nto \GFIX\lk{K}{k}{F}$ be the chosen \O-colimit.
  The components are:
  \begin{align*}
    \ms{fold}_{\lk{K}{k}{F}} &= \dirsup_{n \in \N} \kappa_{n+1} \circ F\kappa_n^p :  F(\GFIX\lk{K}{k}{F}) \to \GFIX\lk{K}{k}{F}\\
    \ms{unfold}_{\lk{K}{k}{F}} &= \dirsup_{n \in \N} F\kappa_n \circ \kappa_{n+1}^p : \GFIX\lk{K}{k}{F} \to F(\GFIX\lk{K}{k}{F}).
  \end{align*}
  Naturality means that for every $(f, \eta) : \lk{K}{k}{F} \to \lk{L}{l}{G}$, the following diagram commutes:
  \begin{equation}%
    \label[diagram]{eq:main:9}%
    \begin{tikzcd}[column sep=4em]%
      F(\GFIX\lk{K}{k}{F}) \ar[r, shift left=0.5ex, "{\ms{fold}_{\lk{K}{k}{F}}}"] \ar[d, swap, "{\UNF(f, \eta)}"]  & \GFIX\lk{K}{k}{F} \ar[l, shift left=0.5ex, "{\ms{unfold}_{\lk{K}{k}{F}}}"]  \ar[d, "{\GFIX(f, \eta)}"]\\%
      G(\GFIX\lk{L}{l}{G}) \ar[r, shift left=0.5ex, "{\ms{fold}_{\lk{L}{l}{G}}}"] & \GFIX\lk{L}{l}{G} \ar[l, shift left=0.5ex, "{\ms{unfold}_{\lk{L}{l}{G}}}"].%
    \end{tikzcd}%
  \end{equation}
\end{proposition}

We can specialize these constructions to produce fixed points that are, in a sense made clear in \cref{sec:canon-fixed-points}, canonical.
Assume $\mb{K}^e$ has an initial object.
We first observe that $[ \mb{K} \lcto \mb{K} ]$ embeds fully and faithfully into $\Links_{\mb{K}}$ via the locally continuous functor that maps objects $F : \mb{K} \lcto \mb{K}$ to the link $\lk{\bot}{\bot}{F}$ and natural transformations $\eta : F \nto G$ to the morphism $(\ms{id}_\bot, \eta)$.

We say that an \O-category $\mb{K}$ has \defin{strict morphisms} if it has zero morphisms and $0_{AB}$ is the least element of $\mb{K}(A, B)$ for all object objects $A$ and $B$.
We say that $\mb{K}$ \defin{supports canonical fixed points} if it has an initial object, strict morphisms, and is \O-cocomplete.
Assume $\mb{K}$ supports canonical fixed points.
Then $\bot$ is also the initial object of $\mb{K}^e$.
We define the \defin{canonical-fixed-point functor} $\FIX : [ \mb{K} \lcto \mb{K} ] \to \mb{K}^e$ as the composition $[ \mb{K} \lcto \mb{K} ] \hookrightarrow \Links_{\mb{K}} \xrightarrow{\GFIX} \mb{K}^e$.
This functor is locally continuous by \cref{prop:main:8}.

\subsection{Parametrized Fixed Points}
\label{sec:param-fixed-points}

In this section, we show that taking parametrized fixed points of functors is functorial.
In particular, we show in \cref{prop:main:4} that for appropriate \O-categories $\mb{D}$ and $\mb{E}$, there exists a locally continuous functor $[\mb{D} \times \mb{E} \lcto \mb{E} ] \to [\mb{D} \lcto \mb{E}^e ]$ that induces an external dagger operation.
This proposition generalizes the mapping $F \mapsto \sfix{F}$ of \cite[Proposition~5.2.7]{abramsky_jung_1995:_domain_theor} from functors on categories of pointed domains closed under bilimits to locally continuous functors on $\mb{O}$-categories.
By also defining the action on natural transformation, it makes the mapping's functorial structure evident.
It specializes the construction of \cite[\S~3]{lehmann_smyth_1977:_data_types} to \O-categories.

Given a functor $F : \mb{D} \times \mb{E} \to \mb{E}$ and an object $D$ of $\mb{D}$, we write $F_D$ for the partial application $\Lambda F D = F(D, {-}) : \mb{E} \to \mb{E}$.

\begin{restatable}{proposition}{PropMainE}%
  \label{prop:main:4}
  Let $\mb{D}$ and $\mb{E}$ be $\mb{O}$-categories.
  Assume $\mb{E}$ supports canonical fixed points.
  The following defines a locally continuous functor:
  \[
    \sfix{({\cdot})} = [\ms{id}_{\mb{D}} \to \FIX] \circ \Lambda : [\mb{D} \times \mb{E} \lcto \mb{E} ] \to [\mb{D} \lcto \mb{E}^e ].
  \]
  Explicitly, $\sfix{F}D = \FIX(F_D)$ is the canonical fixed point of $F(D, -)$.
  Given a natural transformation $\eta : F \nto G$, $\left(\sfix{\eta}\right)_D = \FIX\left((\Lambda \eta)_D\right)$.
\end{restatable}

Let {\CFP} be the full subcategory of {\O} whose objects are \O-categories that support canonical fixed points.
It is 2-cartesian closed~\cite[Theorem~7.3.11]{fiore_1994:_axiom_domain_theor}.

\begin{corollary}
  \label{cor:main:1}
  \Cref{prop:main:4} defines an external dagger operator on {\CFP}.
\end{corollary}

The weak fixed-point identity gives us an analog of \cref{prop:main:3} for solutions to parametrized equations, \ie, to equations of the form $E \cong F(D,E)$.
The fixed point identity $\sfix{F}D = F(D, \sfix{F}D)$ typically does not hold for daggers of functors because the two functors are not equal on the nose.
However, it holds up to natural isomorphism, giving the weak fixed-point identity:

\begin{proposition}[Weak fixed-point identity]%
  \label{prop:32}
  Let $\mb{E}$ and $\mb{D}$ be $\mb{O}$-categories.
  Assume $\mb{E}$ supports canonical fixed points.
  Let $F : \mb{D} \times \mb{E} \to \mb{E}$ be a locally continuous functor.
  There exist natural transformations
  \begin{align*}
    \ms{Unfold}^F &: \sfix F \nto F \circ \langle \ms{id}_{\mb{D}}, \sfix F \rangle\\
    \ms{Fold}^F &: F \circ \langle \ms{id}_{\mb{D}}, \sfix F \rangle \nto \sfix F
  \end{align*}
  that form a natural isomorphism $\sfix F \cong F \circ \langle \ms{id}_{\mb{D}}, \sfix F \rangle$.
  Let $D$ be an object of $\mb{D}$.
  The $D$-components for these natural transformations are given by
  \begin{align*}
    \ms{Unfold}^F_D &= \ms{unfold}_{\lk{\bot}{\bot}{F_D}} : \GFIX\lk{\bot}{\bot}{F_D} \to \UNF\lk{\bot}{\bot}{F_D}\\
    \ms{Fold}^F_D &= \ms{fold}_{\lk{\bot}{\bot}{F_D}} : \UNF\lk{\bot}{\bot}{F_D} \to \GFIX\lk{\bot}{\bot}{F_D}
  \end{align*}
  where $\ms{unfold}$ and $\ms{fold}$ are the natural isomorphisms given by \cref{prop:main:3}.
  The definitions of $\ms{Fold}^F$ and $\ms{Unfold}^F$ are natural in $F$.
  Given any natural transformation $\eta : F \nto G$, the following two squares commute:
  \[
    \begin{tikzcd}[column sep=4em, arrows=Rightarrow]
      \sfix{F}
      \ar[r, "{\ms{Unfold}^F}"]
      \ar[d, swap, "{\sfix\eta}"]
      & F\circ\langle\ms{id}, \sfix{F} \rangle
      \ar[d, "{\eta \ast \langle \ms{id}, \sfix{\eta} \rangle}"]
      &F\circ\langle\ms{id}, \sfix{F} \rangle
      \ar[r, "{\ms{Fold}^F}"]
      \ar[d, swap, "{\eta \ast \langle \ms{id}, \sfix{\eta} \rangle}"]
      & \sfix{F}
      \ar[d, "{\sfix\eta}"]\\
      \sfix{G} \ar[r, "{\ms{Unfold}^G}"]
      & G\circ\langle\ms{id}, \sfix{G} \rangle
      & G\circ\langle\ms{id}, \sfix{G} \rangle
      \ar[r, "{\ms{Fold}^G}"]
      & \sfix{G}
    \end{tikzcd}
  \]
\end{proposition}

To the best of our knowledge, the fact that $\ms{Fold}^F$ and $\ms{Unfold}^F$ are natural in $F$ is new.
It will be key to defining the interpretations of recursive session types in \cref{sec:appl-semant}.

\Cref{prop:main:13} follows easily from \cref{prop:main:4,prop:32}.
It illustrates the action of $\sfix{(\cdot)}$ on natural transformations and gives identities that will be useful in \cref{sec:appl-semant}.
It affirmatively answers the first question of the introduction: the definition of $\sfix{(\cdot)} : [\mb{D} \times \mb{E} \lcto \mb{E} ] \to [\mb{D} \lcto \mb{E}^e ]$ is natural in $\mb{D}$.

\begin{restatable}[Parameter Identity]{proposition}{propMainBD}%
  \label{prop:main:13}%
  Let $\mb{C}$, $\mb{D}$ and $\mb{E}$ be \O-categories and assume $\mb{E}$ supports canonical fixed points.
  Let $F, H : \mb{D} \times \mb{E} \to \mb{E}$ and $G, I : \mb{C} \to \mb{D}$ be locally continuous.
  Set $F_G = F \circ (G \times \ms{id}_E) : \mb{C} \times \mb{E} \to \mb{E}$, and analogously for $H_I$.
  Let $\phi : F \nto H$ and $\gamma : G \nto I$ be natural transformations.~Then
  \begin{align}
    \sfix{F_G} &= \sfix{F} \circ G : \mb{C} \to \mb{E},\label{eq:main:10}\\
    F_G \circ \langle \ms{id}_{\mb{C}}, \sfix{F_G} \rangle &= F \circ \langle \ms{id}_{\mb{D}}, \sfix{F} \rangle \circ G : \mb{C} \to \mb{E},\label{eq:main:11}\\
    \sfix{\left(\phi \ast (\gamma \times \ms{id})\right)} &= \sfix{\phi} \ast \gamma : \sfix{F_G} \nto \sfix{H_I},\label{eq:main:14}\\
    \ms{Fold}^{F_G} &= \ms{Fold}^FG : F_G \circ \langle \ms{id}_{\mb{C}}, \sfix{F_G} \rangle \nto \sfix{F_G},\label{eq:main:12}\\
    \ms{Unfold}^{F_G} &= \ms{Unfold}^FG : \sfix{F_G} \nto F_G \circ \langle \ms{id}_{\mb{C}}, \sfix{F_G} \rangle.\label{eq:main:13}
  \end{align}
\end{restatable}

\section{Canonicity of Fixed Points}
\label{sec:canon-fixed-points}

It is well-known that the fixed points $\FIX(F)$ of \cref{sec:funct-fixed-points} are canonical in the sense that $(\FIX(F), \ms{fold})$ is the \textit{initial} $F$-algebra.
Given a functor $F : \mb{C} \to \mb{C}$, an \defin{$F$-algebra} is a pair $(A,a)$ where $A$ and $a$ are respectively an object and a morphism $FA \to A$ in $\mb{C}$.
A morphism $f : (A, a) \to (B, b)$ of $F$-algebras is a morphism $f : A \to B$ in $\mb{C}$ such that $f \circ a = b \circ Ff$.
Such a morphism is called an \defin{$F$-algebra homomorphism}.
These objects and morphisms form a category $\algc{\mb{C}}{F}$ of $F$-algebras.
The category $\coalgc{\mb{C}}{F}$ of \defin{$F$-coalgebras} is symmetrically defined.
An $F$-coalgebra is a pair $(A,a)$ where $a : A \to FA$.
An $F$-coalgebra homomorphism $(A,a) \to (B, b)$ is a morphism $f : A \to B$ in $\mb{C}$ such that $Ff \circ a = b \circ f$.

\begin{restatable}[{\cites[Lemma~2]{smyth_plotkin_1982:_categ_theor_solut}[Theorem~10.3]{gunter_1992:_seman_progr_languag}}]{proposition}{propEB}
  \label{prop:41}
  Let $\mb{K}$ be an \O-cocomplete \O-category and let $F$ be a locally continuous functor on $\mb{K}$.
  Let $\kappa : \Omega(\bot,\bot,F) \nto \FIX(F)$ be the chosen \O-colimit.
  \begin{enumerate}
  \item \label{item:16} The initial $F$-algebra is $(\FIX(F), \ms{fold})$.
    Given any other $F$-algebra $(A, a)$, the unique $F$-algebra homomorphism $(\FIX(F), \ms{fold}) \to (A, \alpha)$ is the embedding $\phi = \dirsup_{n \in \N} \alpha_n \circ \kappa_n^p$, where $\alpha : \Omega(\bot,\bot,F) \nto A$ is the cocone inductively defined by $\alpha_0 = \bot_A$ and $\alpha_{n + 1} = a \circ F\alpha_n$.
  \item \label{item:17}
    The terminal $F$-coalgebra is $(\FIX(F), \ms{unfold})$.
    Given any other $F$-coalgebra $(B, b)$, the unique $F$-coalgebra homomorphism $(B, b) \to (\FIX(F), \ms{unfold})$ is the projection $\rho = \dirsup_{n \in \N} \kappa_n \circ \beta_n$, where $\beta : B \nto \Omega(\bot,\bot,F)$ is the cone inductively defined by $\beta : \beta_0 = \bot_B^p$ and $\beta_{n + 1} = F\beta_n \circ b$.
  \item \label{item:18}
    Given an $F$-algebra $(C, c)$ where $c$ is an embedding, then $(C, c^p)$ is an $F$-coalgebra and $(\phi, \rho)$ form an e-p-pair.
  \end{enumerate}
\end{restatable}

\Cref{prop:45} generalizes \cref{prop:41} to parametrized fixed points.
Given a horizontal morphism $f : A \times B \to B$ in a 2-cartesian category, an \defin{$f$-algebra}~\cite[Definition~2.3]{bloom_esik_1995:_some_equat_laws_initial} is a pair $(g, u)$ where $g : A \to B$ is a horizontal morphism and $u : f \circ \langle \ms{id}_A, g \rangle \nto g$ is vertical.
An \defin{$f$-algebra homomorphism} $(g,u) \to (h,v)$ is a vertical morphism $w : g \nto h$ such that $w \circ u = v \circ (f \ast \langle \ms{id}_A, w\rangle)$.
These $f$-algebras and $f$-algebra homomorphisms form a category.

\begin{restatable}{proposition}{propEF}
  \label{prop:45}
  Let $\mb{E}$ and $\mb{D}$ be $\mb{O}$-categories, and assume $\mb{E}$ supports canonical fixed points.
  Let $F : \mb{D} \times \mb{E} \to \mb{E}$ be a locally continuous functor.
  Let $\sfix{F} : \mb{D} \to \mb{E}$ be given by \cref{prop:main:4}, and $\ms{Fold}$ and $\ms{Unfold}$ by \cref{prop:32}.
  Then $(\sfix{F}, \ms{Fold})$ and $(\sfix{F}, \ms{Unfold})$ are respectively the initial $F$-algebra and terminal $F$-coalgebra.
  \begin{enumerate}
  \item Given any other $F$-algebra $(G, \gamma)$, the mediating morphism $\phi : \sfix{F} \to G$ is a natural family of embeddings.
    The component $\phi_D$ is the unique $F_D$-algebra homomorphism $(\sfix{F}D, \ms{Fold}_D) \to (GD, \gamma_D)$ given by \cref{prop:41}.\label{item:3}
  \item Given any other $F$-coalgebra $(\Gamma, \gamma)$, the mediating morphism $\rho : \Gamma \to \sfix{F}$ is a natural family of projections.
    The component $\rho_D$ is the unique $F_D$-coalgebra homomorphism $(GD, \gamma_D) \to (\sfix{F}D, \ms{Unfold}_D)$ given by dualizing \cref{prop:41}.\label{item:4}
  \item Given an $F$-algebra $(A, \alpha)$ where $\alpha$ is an embedding in $[\mb{D} \lcto \mb{E}]$, then $(A, \alpha^p)$ is an $F$-coalgebra and $(\phi, \rho)$ form an e-p-pair in $[\mb{D} \lcto \mb{E}]$.\label{item:7}
  \end{enumerate}
\end{restatable}

\Cref{prop:45} presents the converse of a class of external daggers on horizontal morphisms considered in \cite[\S~2.2]{bloom_esik_1995:_some_equat_laws_initial}.
Given a horizontal morphism $f : A \times B \to B$ in a cartesian 2-category, they define $\sfix{f} = g$ where $(g,v)$ is the initial $f$-algebra.
They do not consider the action of this dagger on vertical morphisms.
In contrast, we give a dagger operation that determines initial $f$-algebras.
It induces an action on both horizontal and vertical morphisms.
By \cref{prop:32,prop:main:13}, its action on vertical morphisms coheres with its action on horizontal morphisms.

\section{Conway Identities}
\label{sec:conway-identities}

The dagger operator of \cref{prop:main:4} satisfies the \textit{Conway identities}~\cite{bloom_esik_1995:_some_equat_laws_initial,bloom_esik_1996:_fixed_point_operat} up to isomorphism.
Our interest in the Conway identities stems from the fact that they imply a class of identities useful in the semantics of programming language.

We begin by presenting the Conway identities~\cite{bloom_esik_1995:_some_equat_laws_initial,bloom_esik_1996:_fixed_point_operat}.
Given an external dagger operation $\dagger_{A,B} : \mb{C}(A \times B, B) \to \mb{C}(A, B)$ and a morphism $f : A \times B \to B$, we write $\sfix{f}$ for $\dagger_{A,B}(f)$.
An external dagger $\dagger$ satisfies:
\begin{enumerate}
\item the \defin{parameter identity} or \defin{naturality} if for all $f : B \times C \to C$ and $g : A \to B$, $\sfix{(f \circ (g \times \ms{id}_C))} = \sfix{f} \circ g$.\label[property]{item:main:1}
\item the \defin{composition identity} or \defin{parametrized dinaturality} if for all $f : P \times A \to B$ and $g : P \times B \to A$, $\sfix{(g \circ \langle \pi^{P \times A}_P, f \rangle)} = g \circ \langle \ms{id}_P, \sfix{(f \circ \langle \pi^{P \times B}_P, g \rangle)} \rangle$.\label[property]{item:main:2}
\item the \defin{double dagger identity} or \defin{diagonal property} if for all $f : A \times B \times B \to B$, $\sfix{(\sfix{f})} = \sfix{(f \circ (\ms{id}_A \times \langle \ms{id}_B, \ms{id}_B \rangle))}$.\label[property]{item:main:3}
\item the \defin{abstraction identity} if the following diagram commutes:\label[property]{item:main:4}
  \[
    \begin{tikzcd}[column sep=1.5em,row sep=1em]
      {[A \times B \times C \to C]}
      \ar[rrr, "{[\ms{id}_A \times \langle \pi_B, \ms{ev}_{B, C} \rangle \to \ms{id}_C]}"]
      \ar[d, swap, "{\dagger_{A \times B, C}}"]
      &&
      & {[A \times [B \to C] \times B \to C]}
      \ar[d, "{\Lambda}"]
      \\
      {[A \times B \to C]}
      \ar[r, "{\Lambda}"]
      &
      {[A \to [B \to C]]}
      && {[A \times [B \to C] \to [B \to C]]}
      \ar[ll, swap, "{\dagger_{A,[B \to C]}}"]
    \end{tikzcd}
  \]
\item the \defin{power identities} if for all $f : A \times B \to B$ and $n > 1$, $\sfix{(f^n)} = \sfix{f}$, where $f^n : A \times B \to B$ is inductively defined by $f^0 = \pi^{A \times B}_B$ and $f^{n + 1} = f \circ \langle \pi^{A \times B}_A, f^n \rangle$.\label[property]{item:main:5}
\end{enumerate}

An external dagger satisfies the \defin{Conway identities} if it satisfies \cref{item:main:1,item:main:2,item:main:3,item:main:4}.
\Cref{theorem:main:1} answers the last question of \cref{sec:introduction}.
It is immediate from \cite[Theorem~7.1]{bloom_esik_1995:_some_equat_laws_initial}, \cref{prop:main:13,prop:45}, and the fact that {\CFP} is 2-cartesian closed.

\begin{theorem}
  \label{theorem:main:1}
  The dagger operation of \cref{prop:main:4} satisfies the Conway identities and the power identities up to isomorphism.
\end{theorem}

The Conway identities imply the \defin{pairing identity}, sometimes called Beki{\v{c}}'s identity~\cite[10]{bloom_esik_1996:_fixed_point_operat}.
It is useful for solving systems of parametrized equations, \eg, as we did in \cref{sec:introduction} for the functors defining data types \texttt{even} and \texttt{odd}.

\begin{proposition}[Pairing Identity]
  \label{prop:main:12}
  Let $\mb{A}$, $\mb{B}$, and $\mb{C}$ be small \O-categories, and assume $\mb{B}$ and $\mb{C}$ support canonical fixed points.
  Let $F : \mb{A} \times \mb{B} \times \mb{C} \to \mb{B}$ and $G : \mb{A} \times \mb{B} \times \mb{C} \to \mb{C}$ be locally continuous functors.
  Set $H = \mb{A} \times \mb{B} \xrightarrow{\langle \ms{id}_{\mb{A}}, \sfix{G} \rangle} \mb{A} \times \mb{B} \times \mb{C} \xrightarrow{F} \mb{B}$.
  Then $\sfix{\langle F, G \rangle} = \langle \sfix{G} \circ \langle \ms{id}_{\mb{A}}, \sfix{H} \rangle, \sfix{H} \rangle : \mb{A} \to \mb{B} \times \mb{C}$.
\end{proposition}

\section{Applications to Semantics of Session Types}
\label{sec:appl-semant}

We illustrate our results by applying them to denotational semantics for session-typed languages.
In particular, we show that they are essential both for defining and reasoning about the denotations of recursive session types.
Session types specify communication protocols between processes.
We consider the restricted setting of two processes $S$ (the server) and $C$ (the client).
They independently perform computation and communicate with each other over a wire $c$ (a channel).
This bidirectional communication on $c$ is specified by a session type $A$ that evolves over the course of execution.
We can think of $S$ as a function from its input on $c$ to its output on $c$, and of $C$ as a function from its input on $c$ to its output on $c$.
To accomplish this we imagine $c$ as a pair $(c^-, c^+)$ of wires carrying unidirectional communications: a wire $c^-$ that carries communications from $C$ to $S$, and a wire $c^+$ that carries communication from $S$ to $C$.
This gives rise to the picture \begin{tikzcd}[cramped] S \ar[r, shift left=0.5ex, "{c^+}"] & C \ar[l, shift left=0.5ex, "{c^-}"]\end{tikzcd}.

We interpret a session type $A$ as a Scott domain $\sembr{A}$ whose elements are the bidirectional communications permitted by the protocol $A$.
To interpret the picture, we decompose $\sembr{A}$ into a pair of Scott domains $\sembr{A}^-$ and $\sembr{A}^+$.
The domain $\sembr{A}^+$ contains the left-to-right unidirectional communications on $c^+$ that $A$ permits, and symmetrically for $\sembr{A}^-$.
We then interpret $S$ and $C$ as continuous functions $\sembr{S} : \sembr{A}^- \to \sembr{A}^+$ and $\sembr{C} : \sembr{A}^+ \to \sembr{A}^-$.

The decomposition of $\sembr{A}$ into $\sembr{A}^- \times \sembr{A}^+$ introduces ``semantic junk''.
Indeed, $\sembr{A}^- \times \sembr{A}^+$ contains many pairs $(a^-, a^+)$ of unidirectional communications that do not correspond to bidirectional communications $a \in \sembr{A}$.
We use an embedding $\ppi{A} : \sembr{A} \to \sembr{A}^- \times \sembr{A}^+$ to pick out the pairs $(a^-, a^+)$ that correspond to genuine bidirectional communications.

We illustrate this semantic approach by giving interpretations to recursive session types.
We note that this approach is also applicable to a rich class of session types including internal and external choice, channel transmission, synchronization, etc.\@
Due to space constraints, we do not consider their interpretations~here.

Assume recursive session types $\Trec{\alpha}{A}$ are formed using the following rules.
The judgment $\Xi \vdash {A}$ means that $A$ is a session type in the presence type variables $\Xi = \alpha_1, \dotsc, \alpha_n$ ($n \geq 0$).
\[
  \infer[\rn{CVar}]{\Xi,\alpha\vdash{\alpha}}{}
  \quad
  \infer[\rn{C$\rho$}]{
    \Xi \vdash {\Trec{\alpha}{A}}
  }{
    \Xi, {\alpha} \vdash {A}
  }
\]

To handle open types, we generalize from triples of domains and a single embedding to a 2-cell whose components are embeddings.
Let $\mb{M}$ be the \O-category of Scott domains where morphisms are strict continuous meet-preserving functions.
It supports canonical fixed points.
Let $\sembr{\alpha_1,\dotsc,\alpha_n}$ be the category $\prod_{i = 1}^n \mb{M}$.
Then $\Xi \vdash {A}$ denotes a {2-cell} $\ppi{\Xi \vdash {A}} : \sembr{\Xi \vdash {A}} \nto \sembr{\Xi \vdash {A}}^- \times \sembr{\Xi \vdash {A}}^+ : \sembr{\Xi} \to \mb{M}$ where each component of the natural transformation $\ppi{\Xi\vdash A}$ is an embedding in $\mb{M}$.
The interpretation of $\Xi \vdash A$ is defined by induction on its derivation.

The functors interpreting \rn{CVar} are projection of the $\alpha$ component:
\[
  \sembr{\Xi,\alpha\vdash{\alpha}} = \sembr{\Xi,\alpha\vdash{\alpha}}^- = \sembr{\Xi,\alpha\vdash{\alpha}}^+ = \pi_{n+1}, \quad \ppi{\Xi,\alpha\vdash{\alpha}} = \langle \ms{id}, \ms{id} \rangle.
\]
The functors interpreting $\Xi \vdash \Trec{\alpha}{A}$ are defined using \cref{prop:main:4}:
\[
  \sembr{\Xi \vdash {\Trec{\alpha}{A}}} = \sfix{\sembr{\Xi, {\alpha} \vdash {A}}}, \quad \sembr{\Xi \vdash {\Trec{\alpha}{A}}}^p = \sfix{\left(\sembr{\Xi, {\alpha} \vdash {A}}^p\right)} \; (p \in \{{-},{+}\}.
\]
Abbreviate $\Xi \vdash \Trec{\alpha}{A}$ by $\vdash \Trec{\alpha}{A}$ and $\Xi, \alpha \vdash A$ by $\alpha \vdash A$.
Set $\ppi{\alpha \vdash A}^- = \pi_1 \circ \ppi{\alpha \vdash A}$ and $\ppi{\alpha \vdash A}^+ = \pi_2 \circ \ppi{\alpha \vdash A}$.
Instantiating $\eta$ in the right diagram of \cref{prop:32} by $\ppi{\alpha \vdash A}^p$ for $p \in \{{-},{+}\}$ and expanding the definition of the horizontal composition $\eta \ast \langle \ms{id}, \sfix\eta \rangle$ gives:
\[
  \begin{tikzcd}[column sep=13em]
    \sembr{\alpha \vdash A} \circ \langle\ms{id}_{\sembr{\Xi}}, \sembr{\vdash \Trec{\alpha}{A}} \rangle
    \ar[r, Rightarrow, "{\ms{Fold}^{\sembr{\alpha \vdash A}}}"]
    \ar[d, swap, Rightarrow, "{\sembr{\alpha \vdash A}\langle \ms{id}, \sfix{(\ppi{\alpha \vdash A}^p)} \rangle}"]
    &
    \sembr{\vdash \Trec{\alpha}{A}}
    \ar[d, Rightarrow, "{\sfix{(\ppi{\alpha \vdash A}^p)}}"]
    \\
    \sembr{\alpha \vdash A} \circ \langle\ms{id}_{\sembr{\Xi}}, \sembr{\vdash \Trec{\alpha}{A}}^p \rangle
    \ar[r, Rightarrow, "{\ms{Fold}^{\sembr{\vdash \Trec{\alpha}{A}}^p} \circ \ppi{\alpha \vdash A}^p\langle\ms{id}, \sembr{\vdash \Trec{\alpha}{A}}^p \rangle}"]
    &
    \sembr{\vdash \Trec{\alpha}{A}}^p
  \end{tikzcd}
\]
The category of $\sembr{\alpha \vdash A}$-algebras has products, so there exists a mediating morphism $\langle \sfix{(\ppi{\alpha \vdash A}^-)}, \sfix{(\ppi{\alpha \vdash A}^+)} \rangle : \sembr{\vdash \Trec{\alpha}{A}} \nto \sembr{\vdash \Trec{\alpha}{A}}^- \times \sembr{\vdash \Trec{\alpha}{A}}^+$.
It is a natural family of embeddings by \cref{prop:45}, so we define:
\[
  \ppi{\vdash \Trec{\alpha}{A}} = \langle \sfix{(\ppi{\alpha \vdash A}^-)}, \sfix{(\ppi{\alpha \vdash A}^+)} \rangle.
\]

\Cref{prop:main:13} implies that these denotations respect substitution (\cf \cite[Proposition~36]{kavanagh_2020:_domain_seman_higher}):

\begin{restatable}{proposition}{propBF}
  \label{prop:main:15}
  If $\alpha_1,\dotsc,\alpha_n \vdash  A$ and $\Theta \vdash {B_i}$ for $1 \leq i \leq n$, then for $\genbr{\cdot}$ ranging in $\sembr{\cdot}$, $\sembr{\cdot}^-$ and $\sembr{\cdot}^+$, $\genbr{{\Theta \vdash {[\vec{B}/\vec{\alpha}] A}}} = \genbr{{\Xi \vdash A}} \circ \langle \genbr{{\Theta \vdash A_i}} \rangle_{1 \leq i \leq n}$.
  Moreover, $\ppi{{\Theta \vdash {[\vec{B}/\vec{\alpha}] A}}} = \ppi{{\Xi \vdash A}} \ast \langle \ppi{{\Theta \vdash A_i}} \rangle_{1 \leq i \leq n}$.
\end{restatable}

\Cref{prop:main:15,prop:32} imply that the above denotations respect syntactic folding and unfolding of recursive up-to-isomorphism, \ie,
\[
  \sembr{\Xi \vdash \Trec{\alpha}{A}} \cong \sembr{\Xi, \alpha \vdash A} \circ \langle \ms{id}_{\sembr{\Xi}}, \ppi{\Xi \vdash \Trec{\alpha}{A}} \rangle = \sembr{\Xi \vdash [\Trec{\alpha}{A}/\alpha] A}.
\]

\section{Related Work}
\label{sec:related-work}

\textcite{lehmann_smyth_1977:_data_types,lehmann_smyth_1981:_algeb_specif_data_types} introduced the idea of interpreting datatypes as initial fixed points of $\omega$-cocontinuous functors on $\omega$-cocomplete categories with initial objects.
Their constructions for initial and parametrized fixed points of functors generalize $\FIX$ and \cref{prop:main:4} to their setting.
By specializing their construction to \O-categories, we give an explicit recipe for constructing these fixed points and their associated initial algebras.

\textcite{wand_1977:_fixed_point_const} introduced the definitions of \O-categories and locally continuous functors.
\textcite{smyth_plotkin_1977:_categ_theor_solut,smyth_plotkin_1982:_categ_theor_solut} introduced \O-(co)limits and generalized Scott's limit-colimit coincidence theorem to \O-categories.

Iteration theories~\cite{bloom_esik_1993:_iterat_theor} were introduced to study the syntax and semantics of flowchart algorithms.
Iteration theories are defined in terms of a dagger operation.
\textcite{bloom_esik_1996:_fixed_point_operat} studied external dagger operations on cartesian closed categories and showed that many of the categories used in semantics, the least fixed point operator induces a dagger operation satisfying the Conway identities.
They generalized this work to cartesian closed 2-categories in \cite{bloom_esik_1995:_some_equat_laws_initial}  and gave sufficient conditions for a dagger on horizontal morphisms to satisfy the Conway identities.
They did not consider the action of daggers on vertical~morphisms.

\textcite{simpson_plotkin_2000:_compl_axiom_categ} gave an axiomatic treatment of dagger operations satisfying Conway identities.
They gave a purely syntactic account of free iteration theories.
They give a precise characterization of the circumstances in which the iteration theory axioms are complete for categories with an iteration operator.

\textcite{fiore_1994:_axiom_domain_theor} investigated axiomatic categorical domain theory for application to the denotational semantics of deterministic programming languages.
Chapter~6 defines a dagger operation on functors between certain algebraically complete \O-categories.
This dagger satisfies the parameter identity on functors, \ie, it satisfies \cref{eq:main:10} above.
Under certain conditions, this dagger operation extends to the functor given by \cref{prop:main:4} (\cf~\cite[130--131]{fiore_1994:_axiom_domain_theor}).
Our category {\CFP} appears as the category $\mb{Kind}$~\cite[Definition~7.3.11]{fiore_1994:_axiom_domain_theor}.

\textcite{honda_1993:_types_dyadic_inter,takeuchi_1994:_inter_based_languag} introduced session types to describe sessions of interaction.
\textcite{caires_pfenning_2010:_session_types_intuit_linear_propos} observed a proofs-as-programs correspondence between the session-typed $\pi$-calculus and intuitionistic linear logic.
\textcite{kavanagh_2020:_domain_seman_higher} gave the first denotational semantics for a language with session-typed concurrency and general recursion.

\section{Conclusion and Acknowledgments}
\label{sec:concl-ackn}

We gave a functorial dagger operation that satisfies the Conway identities and that is defined both on functors and natural transformations.
We also proved various order-theoretic properties about the dagger operation.
In \cref{sec:appl-semant}, we saw that the Conway identities and the dagger operation's order-theoretic properties were essential for defining the semantics of recursive session types.

This work is funded in part by a Natural Sciences and Engineering Research Council of Canada Postgraduate Scholarship.
The author thanks Stephen Brookes and Frank Pfenning for their comments.

\printbibliography

\appendix

\section{General Results on \O-categories}
\label{sec:general-results-o-cats}

In this section, we present various results concerning e-p-pairs, locally continuous functors, and colimits in \O-categories.
Many of these results are standard and we present them only for ease of reference.

An \O-category $\mb{K}$ has \defin{locally determined $\omega$-colimits of embeddings}~\cite[Definition~7]{smyth_plotkin_1982:_categ_theor_solut} if for all $\omega$-chains $\Delta$ in $\mb{K}^e$ and cocones $\kappa : \Delta \nto A$ in $\mb{K}^e$, $\kappa$ is colimiting in $\mb{K}^e$ if and only if $\kappa$ is an \O-colimit.

We frequently need to deal with cocones, morphisms of cocones, and colimits.
It is useful to introduce some terminology to make their structure explicit.

There exists~\cite[Definition~3.1.5]{riehl_2016:_categ_theor_contex} a functor $\Cone(F, {-}) : \mb{C} \to \mb{Set}$ taking an object $C$ of $\mb{C}$ to the set of cocones under $F$ with nadir $C$.
Given a morphism $f : C \to C'$ and a cocone $(\lambda : F \nto C) \in \Cone(F, C)$, $\Cone(F, f)(\lambda) = f \circ \lambda$.

Given a functor $F : \mb{C} \to \mb{Set}$, the \defin{category of elements} $\cel F$ has as objects pairs $(c, C)$ where $C$ is an object of $\mb{C}$ and $c \in FC$.
A morphism $f : (c, C) \to (c', C')$ is a morphism $f : C \to C'$ in $\mb{C}$ such that $F(f)(c) = c'$.

Given a diagram $F : J \to \mb{C}$, the \defin{category of cocones on $F$} is the category of elements $\cel{\Cone(F, {-})}$.
The \defin{colimit} of $F$ is defined to be the initial object of $\cel{\Cone(F, {-})}$~\cite[Definition~3.1.6]{riehl_2016:_categ_theor_contex}.

\subsection{Locally Continuous Functors}
\label{sec:locally-cont-funct}

\begin{lemma}[{\cites[Proposition~5.2.4]{abramsky_jung_1995:_domain_theor}[Lemma~4 and Theorem~3]{smyth_plotkin_1982:_categ_theor_solut}}]
  \label{lemma:15}
  Let $F : \mb{K} \to \mb{L}$ be a locally continuous functor on $\mb{O}$-categories $\mb{K}$ and $\mb{L}$.
  Then $F$ preserves embeddings and projections, \ie, $F$ restricts to functors $F^e : \mb{K}^e \to \mb{L}^e$ and $F^p : \mb{K}^p \to \mb{L}^p$.
  Moreover, $F(e)^p = F\left(e^p\right)$ and $F(p)^e = F\left(p^e\right)$.
  The restriction $F^e$ is $\omega$-cocontinuous when both $\mb{K}$ and $\mb{L}$ have locally determined $\omega$-colimits of embeddings.
\end{lemma}

\begin{proof}
  We begin by showing that $F$ preserves embeddings and projections.
  Let $C$ and $D$ be arbitrary objects and let $e : C \to D$ be an embedding with associated projection $p : D \to C$.
  Then $p \circ e = \ms{id}_C$, so by functoriality,
  \[
    F(p) \circ F(e) = F(p \circ e) = F(\ms{id}_C) = \ms{id}_{FC}.
  \]
  By definition, $F$ induces a continuous map $[D \to C] \to [FD \to FC]$, and continuous maps are monotone.
  Because $e \circ p \sqsubseteq \ms{id}_D$, we get by monotonicity
  \[
    F(e) \circ F(p) = F(e \circ p) \sqsubseteq F(\ms{id}_D) = \ms{id}_{FD}.
  \]
  We conclude that $F(e)$ is an embedding with associated projection $F(p)$.
  Because embeddings and projections uniquely determine each other, $F(e)^p = F(p) = F(e^p)$ and $F(p)^e = F(e) = F(p^e)$.
  If $\mb{K}$ and $\mb{L}$ have locally determined $\omega$-colimits of embeddings, then the restriction $F^e$ is $\omega$-cocontinuous by \cite[Theorem~3]{smyth_plotkin_1982:_categ_theor_solut}.\qed
\end{proof}

\subsection{Cartesian Closure of \O}
\label{sec:cartesian-closure-o}

\begin{proposition}
  \label{prop:main:7}
  The category {\O} is a cartesian 2-category~\cite[96]{bloom_esik_1995:_some_equat_laws_initial} and this structure is inherited from $\mb{Cat}$.
  Explicitly, its terminal object $\mb{1}$ is the one-object category.
  The categorical product is given by the product of categories~\cite[Definition~1.3.12]{riehl_2016:_categ_theor_contex}.
\end{proposition}

\begin{proof}
  Its terminal $\mb{1}$ is the one-object category.
  The homset $\mb{1}(\bullet, \bullet)$ is the dcpo $\{ \ms{id}_{\bullet} \}$.
  So $\mb{1}$ is an \O-category.

  The product structure on {\O} is given by the product of categories~\cite[Definition~1.3.12]{riehl_2016:_categ_theor_contex}.
  Let $\mb{A}$ and $\mb{B}$ be arbitrary small {\O} categories.
  The product of small categories is again small, so $\mb{A} \times \mb{B}$ is small.
  We claim that it is an \O-category.
  Every homset is a dcpo.
  By definition of product categories, $(\mb{A} \times \mb{B})((A, B), (A', B')) = \mb{A}(A,  A') \times \mb{B}(B, B')$.
  The dcpo structure on $(\mb{A} \times \mb{B})((A, B), (A', B'))$ is given by the product of dcpos $\mb{A}(A, A') \times \mb{B}(B, B')$.
  Composition of morphisms in $\mb{A} \times \mb{B}$ is given component-wise, and so is again continuous.
  So $\mb{A} \times \mb{B}$ is an \O-category.

  To show that this product is the categorical product in \O, we must show that the projection morphisms and that the mediating 2-cell exists in \O.
  Consider the projection functor $\pi_{\mb{A}} : \mb{A} \times \mb{B} \to \mb{A}$.
  Its action on morphisms $(A, B) \to (A', B')$ is given by the projection $\mb{A}(A, A') \times \mb{B}(B, B') \to \mb{A}(A, A')$ in $\mb{DCPO}$.
  This projection is continuous, so we conclude that $\pi_{\mb{A}}$ is locally continuous.
  A symmetric argument gives that $\pi_{\mb{B}} : \mb{A} \times \mb{B} \to \mb{B}$ is locally continuous.

  Now let $\mb{C}$ be a small \O-category and consider 2-cells $\alpha : A_1 \nto A_2 : \mb{C} \to \mb{A}$ and $\beta : B_1 \nto B_2 : \mb{C} \to \mb{B}$.
  We must show that there exists a unique 2-cell $\langle \alpha, \beta \rangle : \mb{C} \to \mb{A} \times \mb{B}$ such that $\pi_{\mb{A}} \circ \langle \alpha, \beta \rangle = \alpha$ and $\pi_{\mb{B}} \circ \langle \alpha, \beta \rangle = \beta$.
  We first show that the horizontal morphisms $\langle A_i, B_i \rangle : \mb{C} \to \mb{A} \times \mb{B}$ exist for $i = 1, 2$.
  They are given by $\langle A_i, B_i \rangle C = (A_iC, B_iC)$ for objects $C$ of $\mb{C}$, and $\langle A_i, B_i \rangle f = (A_if, B_if)$ for morphisms $f : C \to C'$ of $\mb{C}$.
  The map
  \[
    f \mapsto (A_if, B_if) : \mb{C}(C, C') \to \mb{A}(A_iC, A_iC') \times \mb{B}(B_iC, B_iC')
  \]
  is exactly the mediating morphism given by the product of homsets in $\mb{DCPO}$, so it is continuous.
  We conclude that $\langle A_i, B_i \rangle$ is locally continuous, so exists in \O.
  The 2-cell $\langle \alpha, \beta \rangle$ is then the natural transformation $\langle \alpha, \beta \rangle_C = \langle \alpha_C, \beta_C \rangle : (A_1C, B_1C) \to (A_2C, B_2C)$.
  Its uniqueness is inherited from $\mb{Cat}$.
  We conclude that {\O} is a cartesian 2-category.\qed
\end{proof}

\begin{proposition}
  \label{prop:main:11}
  If $\mb{A}$ and $\mb{B}$ are \O-categories and $\mb{B}$ is small, then $[\mb{A} \lcto \mb{B}]$ is a small \O-category.
\end{proposition}

\begin{proof}
  The functor category $[\mb{A} \to \mb{B}]$ is small whenever $\mb{B}$ is small.
  The category $[\mb{A} \lcto \mb{B}]$ is a subcategory of $[\mb{A} \to \mb{B}]$, so it too is small.
  Let $F, G : \mb{A} \to \mb{B}$ be two locally continuous functors.
  Then the homset $[\mb{A} \lcto \mb{B}](F,G)$ is again a dcpo.
  Indeed, given two natural transformations $\eta, \rho : F \nto G$, $\eta \sqsubseteq \rho$ if $\eta_A \sqsubseteq \rho_A$ in $\mb{B}(FA, GA)$ for all objects $A$ of $\mb{A}$.
  Directed suprema of directed sets in $[\mb{A} \lcto \mb{B}](F,G)$ are given component-wise, \ie, if $\Gamma$ is a directed set in $[\mb{A} \lcto \mb{B}](F,G)$, then $(\dirsup \Gamma)_A = \dirsup_{\gamma \in \Gamma} \gamma_A$.
  So $[\mb{A} \lcto \mb{B}]$ is an \O-category.\qed
\end{proof}

\begin{lemma}
  \label{lemma:main:6}
  Let $\mb{A}$, $\mb{B}$, and $\mb{C}$ be \O-categories and assume $\mb{B}$ and $\mb{C}$ are small.
  The composition functor ${\circ} : [\mb{B} \lcto \mb{C}] \times [\mb{A} \lcto \mb{B}] \to [\mb{A} \lcto \mb{C}]$ given by ${\circ}(F, G) = FG$ on objects and ${\circ}(\eta, \rho) = \eta \ast \rho$ on morphisms is locally continuous.
\end{lemma}

\begin{proof}
  Let $F_1, F_2 : \mb{B} \to \mb{C}$ and $G_1, G_2 : \mb{A} \to \mb{B}$ be functors.
  Let $\Phi$ and $\Gamma$ respectively be directed sets of natural transformations in the dcpos ${[\mb{B} \lcto \mb{C}](F_1, F_2)}$ and ${[\mb{A} \lcto \mb{B}](G_1, G_2)}$.
  We must show that $\dirsup {\circ}(\Phi, \Gamma) = {\circ}(\dirsup \Phi, \dirsup \Gamma)$.
  Because $\mb{C}$ is an \O-category, composition in $\mb{C}$ is continuous.
  We use the fact that $F_1$ is locally continuous to compute that:
  \begin{align*}
    &\dirsup {\circ}(\Phi, \Gamma)\\
    &= \dirsup_{(\phi, \gamma) \in \Phi \times \Gamma} {\circ}(\phi, \gamma)\\
    &= \dirsup_{\phi \in \Phi} \dirsup_{\gamma \in \Gamma} {\circ}(\phi, \gamma)\\
    &= \dirsup_{\phi \in \Phi} \dirsup_{\gamma \in \Gamma} \phi \ast \gamma\\
    &= \dirsup_{\phi \in \Phi} \dirsup_{\gamma \in \Gamma} \gamma G_2 \circ F_1\phi\\
    &= \left(\dirsup_{\gamma \in \Gamma} \gamma G_2\right) \circ \left(\dirsup_{\phi \in \Phi} F_1\phi\right)\\
    &= \left(\dirsup_{\gamma \in \Gamma} \gamma \right)G_2 \circ F_1\left(\dirsup_{\phi \in \Phi}\phi\right)\\
    &= \left(\dirsup_{\gamma \in \Gamma} \gamma \right) \ast \left(\dirsup_{\phi \in \Phi}\phi\right)\\
    &= {\circ}(\dirsup \Phi, \dirsup \Gamma).
  \end{align*}
  This is what we wanted to show.\qed
\end{proof}

\begin{lemma}
  \label{lemma:main:7}
  Let $\mb{A}$ and $\mb{B}$ be \O-categories and assume $\mb{B}$ is small.
  The evaluation functor $\ms{ev}_{\mb{A},\mb{B}} : [\mb{A} \lcto \mb{B}] \times \mb{A} \to \mb{B}$ is locally continuous.
  It is given by $(F, A) \mapsto FA$ on objects and $(\eta, f) \mapsto \eta \ast f$ on morphisms.
\end{lemma}

\begin{proof}
  Let $(F, A)$ and $(G, B)$ be two objects in $[\mb{A} \to \mb{B}] \times \mb{A}$.
  A morphism $(F, A) \to (G, B)$ is a pair $(\eta, f)$ where $\eta : F \nto G$ and $f : A \to B$.
  Then $\ms{ev}_{\mb{A},\mb{B}}(\eta, f) : FA \to GB$ is given by $\eta \ast f = \eta_B \circ Ff$.
  The functor $F$ is locally continuous, and composition in $\mb{B}$ is continuous, so we conclude that this mapping of morphisms is continuous.
  So $\ms{ev}_{\mb{A},\mb{B}}$ is locally continuous.\qed
\end{proof}

\begin{lemma}
  \label{lemma:main:5}
  Let $\mb{A}$, $\mb{B}$, and $\mb{C}$ be \O-categories and assume $\mb{C}$ is small.
  The functor $\Lambda : [\mb{A} \times \mb{B} \lcto \mb{C}] \to [\mb{A} \lcto [\mb{B} \lcto \mb{C}]]$ is locally continuous, where
  \begin{align*}
    \Lambda F A B &= F(A, B)\\
    \Lambda F A(f : B \to B') &= F(\ms{id}_A, f) : F(A, B) \to F(A, B')\\
    \Lambda F (f : A \to A')_B &= F(f, \ms{id}_B) : F(A, B) \to F(A', B)\\
    \left(\left(\Lambda(\eta : F \nto G)\right)_A\right)_B &= \eta_{(A,B)} : \Lambda FAB \to \Lambda GAB.
  \end{align*}
\end{lemma}

\begin{proof}%
  \begingroup%
  \newcommand{\ABC}{\mb{A} \times \mb{B} \to \mb{C}}%
  \newcommand{\ATBC}{\mb{A} \to [\mb{B} \lcto \mb{C}]}%
  \newcommand{\BC}{\mb{B} \to \mb{C}}%
  We begin by checking that $\Lambda$ is well-defined on objects.
  Let $F : \ABC$ be locally continuous and $A$ be an object of $\mb{A}$.
  Then $\Lambda F A = F(A, {-)} : \BC$ is clearly a functor.
  We must show that $\Lambda FA : \BC$ it is locally continuous.
  Let $D \subseteq \mb{B}(B, B')$ be a directed set.
  We must show that $\Lambda F A(\dirsup D) = \dirsup_{d \in D} \Lambda FAd$.
  But this is obvious because $F$ is locally continuous:
  \[
    \Lambda F A(\dirsup D) = F(\ms{id}, \dirsup D) = \dirsup_{d \in D} F(\ms{id}, d) \dirsup_{d \in D} \Lambda FAd.
  \]
  So $\Lambda F A : \BC$ is a locally continuous functor.

  Next we show that $\Lambda F : \ATBC$ is a locally continuous functor.
  It is well-defined on objects by the above.
  We check that it is well-defined on morphisms.
  Let $f : A \to A'$ be arbitrary in $\mb{A}$.
  We must show that $\Lambda F f : \Lambda F A \nto \Lambda F A'$ is a natural transformation.
  Let $g : B \to B'$ be arbitrary in $\mb{B}$.
  We must show that the following diagram commutes:
  \[
    \begin{tikzcd}[column sep=3em]
      \Lambda F A B
      \ar[r, "{(\Lambda F f)_B}"]
      \ar[d, swap, "{\Lambda F A g}"]
      &
      \Lambda F A' B
      \ar[d, "{\Lambda F A' g}"]\\
      \Lambda F A B'
      \ar[r, "{(\Lambda F f)_B'}"]
      &
      \Lambda F A' B'
    \end{tikzcd}
  \]
  We recognize this as the following diagram, which commutes by the functoriality of $F$:
  \[
    \begin{tikzcd}[column sep=3em]
      F(A, B)
      \ar[r, "{F(f, \ms{id}_B)}"]
      \ar[d, swap, "{F(\ms{id}_A, g)}"]
      &
      F(A', B)
      \ar[d, "{F(\ms{id}_{A'}, g)}"]\\
      F(A, B')
      \ar[r, "{F(f, \ms{id}_{B'})}"]
      &
      F(A', B').
    \end{tikzcd}
  \]
  So $\Lambda F f$ is natural.
  We must now show that $\Lambda F$ is locally continuous.
  Let $D \subseteq \mb{A}(A, A')$ be a directed set.
  We must show that $\Lambda F(\dirsup D) = \dirsup_{d \in D} \Lambda F d$.
  Again, this is obvious because $F$ is locally continuous.
  Let $B$ be an arbitrary object of $\mb{B}$, then:
  \[
    \left(\Lambda F(\dirsup D)\right)_B = F(\dirsup D, \ms{id}_B) = \dirsup_{d \in D} F(d, \ms{id}_B) = \dirsup_{d \in D} \left(\Lambda F d\right)_B = \left(\dirsup_{d \in D} \Lambda F d\right)_B
  \]
  Because $B$ was an arbitrary component, we conclude $\Lambda F(\dirsup D) = \dirsup_{d \in D} \Lambda F d$, \ie, that $\Lambda F$ is locally continuous.

  It follows that $\Lambda$ is well-defined on objects: if $F : \ABC$ is locally continuous, then $\Lambda F$ is an object of $[\mb{A} \lcto [\mb{B} \lcto \mb{C}]]$.

  Next, we show that $\Lambda$ is well-defined on morphisms.
  Let $\eta : F \nto G$ be a natural transformation between two functors $F, G : \ABC$.
  We want to show that $\Lambda \eta : \Lambda F \nto \Lambda G$ is a natural transformation.
  Let $A$ be an object of $\mb{A}$, then must first check that the component $(\Lambda \eta)_A : \Lambda F A \nto \Lambda G A$ is a natural transformation.
  Let $B$ be an object of $\mb{B}$, then $((\Lambda \eta)_A)_B : \Lambda FAB \to \Lambda GAB$ is a morphism in $\mb{C}$.
  Indeed, $((\Lambda \eta)_A)_B = \eta_{(A, B)} : F(A, B) \to G(A, B)$, and $\Lambda FAB = F(A,B)$ and $\Lambda GAB = G(A, B)$.
  We show that $(\Lambda \eta)_A$ is natural.
  Let $g : B \to B'$ be an arbitrary morphism in $\mb{B}$.
  We show that the following square commutes:
  \[
    \begin{tikzcd}[column sep=4em]
      \Lambda F A B
      \ar[r, "{((\Lambda \eta)_A)_B}"]
      \ar[d, swap, "{\Lambda F A g}"]
      &
      \Lambda G A B
      \ar[d, "{\Lambda F A g}"]\\
      \Lambda F A B'
      \ar[r, "{((\Lambda \eta)_{A})_{B'}}"]
      &
      \Lambda G A B'
    \end{tikzcd}
  \]
  This is exactly the following square, which commutes by the naturality of $\eta$:
  \[
    \begin{tikzcd}[column sep=4em]
      F(A, B)
      \ar[r, "{\eta_{(A, B)}}"]
      \ar[d, swap, "{F(\ms{id}_A, g)}"]
      &
      F(A', B)
      \ar[d, "{F(\ms{id}_{A'}, g)}"]\\
      F(A, B')
      \ar[r, "{\eta_{(A, B')}}"]
      &
      F(A', B').
    \end{tikzcd}
  \]
  So $(\Lambda \eta)_A : \Lambda F A \nto \Lambda G A$ is a natural transformation.
  Now we must show that $\Lambda \eta : \Lambda F \nto \Lambda G$ is natural.
  Let $f : A \to A'$ be arbitrary in $\mb{A}$.
  We must show that the following square commutes:
  \[
    \begin{tikzcd}[column sep=4em, arrows=Rightarrow]
      \Lambda F A
      \ar[r, "{(\Lambda \eta)_A}"]
      \ar[d, swap, "{\Lambda F f}"]
      &
      \Lambda G A
      \ar[d, "{\Lambda G f}"]\\
      \Lambda F A'
      \ar[r, "{(\Lambda \eta)_{A'}}"]
      &
      \Lambda G A'
    \end{tikzcd}
  \]
  Natural transformations are equal if and only if they agree in all components, so the above square commutes if and only if the following square commutes for all objects $B$ of $\mb{B}$:
  \[
    \begin{tikzcd}[column sep=4em]
      \Lambda F A B
      \ar[r, "{((\Lambda \eta)_A)_B}"]
      \ar[d, swap, "{(\Lambda F f)_B}"]
      &
      \Lambda G A B
      \ar[d, "{(\Lambda G f)_B}"]\\
      \Lambda F A' B
      \ar[r, "{((\Lambda \eta)_{A'})_B}"]
      &
      \Lambda G A' B
    \end{tikzcd}
  \]
  We recognize this square as the following square, which commutes by naturality of $\eta$:
  \[
    \begin{tikzcd}[column sep=4em]
      F(A, B)
      \ar[r, "{\eta_{(A, B)}}"]
      \ar[d, swap, "{F(f, \ms{id}_B)}"]
      &
      G(A, B)
      \ar[d, "{G(f, \ms{id}_B)}"]\\
      F(A', B)
      \ar[r, "{\eta_{(A', B)}}"]
      &
      G(A', B)
    \end{tikzcd}
  \]
  We conclude that $\Lambda \eta : \Lambda F \nto \Lambda G$ is natural.

  The mapping $\Lambda$ clearly preserves identities and respects composition, so we conclude that $\Lambda$ is a functor.

  We now show that $\Lambda$ is locally continuous.
  Let $H \subseteq [\mb{A} \times \mb{B} \lcto \mb{C}](F, G)$ be a directed set of natural transformations.
  We must show that $\Lambda (\dirsup H) = \dirsup_{\eta \in H} \Lambda \eta$.
  The ordering on natural transformations is determined component-wise, it is sufficient to check for all objects $A$ in $\mb{A}$ that $\left(\Lambda \dirsup H\right)_A = \left(\dirsup_{\eta \in H} \Lambda \eta\right)_A$.
  To show this, it is sufficient to show for all objects $B$ in $\mb{B}$ that $\left(\left(\Lambda \dirsup H\right)_A\right)_B = \left(\left(\dirsup_{\eta \in H} \Lambda \eta\right)_A\right)_B$.
  But this is obvious:
  \[
    \left(\left(\Lambda \dirsup H\right)_A\right)_B = \left(\dirsup H\right)_{(A, B)} = \dirsup_{\eta \in H} \eta_{(A, B)} = \left(\left(\dirsup_{\eta \in H} \Lambda \eta\right)_A\right)_B.
  \]
  We conclude that $\Lambda$ is a locally continuous functor.\qed
  \endgroup
\end{proof}

\begin{lemma}
  \label{lemma:main:8}
  Finally, we must show that for each 2-cell $\alpha : F \nto G : \mb{A} \times \mb{B} \to \mb{C}$, there exists a unique 2-cell $\Lambda \alpha : \mb{A} \to [\mb{B} \lcto \mb{C}]$ such that the following diagram commutes:
  \begin{equation}
    \label[diagram]{eq:main:21}
    \begin{tikzcd}[column sep=4em]
      \mb{A} \times \mb{B} \ar[d, swap, Rightarrow, end anchor={[yshift=-1ex]north}, "{(\Lambda \alpha) \times \ms{id}_{\mb{B}}}"] \ar[dr, Rightarrow, "{\alpha}"] &\\
      {[ \mb{B} \lcto \mb{C} ] \times \mb{B}} \ar[r, swap, "{\ms{ev}_{\mb{B},\mb{C}}}"] & \mb{C}
    \end{tikzcd}
  \end{equation}
\end{lemma}

\begin{proof}
  This requires showing that $\Lambda F$ and $\Lambda G$ are both locally locally continuous functors $\mb{A} \to [\mb{B} \lcto \mb{C}]$ and that:
  \begin{align}
    \ms{ev}_{\mb{B},\mb{C}} \circ ((\Lambda F) \times \ms{id}_{\mb{B}}) &= F\label{eq:main:15}\\
    \ms{ev}_{\mb{B},\mb{C}} \circ ((\Lambda G) \times \ms{id}_{\mb{B}}) &= G\label{eq:main:16}\\
    \ms{ev}_{\mb{B},\mb{C}} \circ ((\Lambda \alpha) \times \ms{id}_{\mb{B}}) &= \alpha.\label{eq:main:17}
  \end{align}
  The functors $\Lambda F$ and $\Lambda G$ are locally continuous and have the right type by \cref{lemma:main:5}.
  By \cref{lemma:main:5}, we also have that $\alpha : \Lambda F \nto \Lambda G$ is a natural transformation.
  So $\Lambda \alpha : \Lambda F \nto \Lambda G$ is a 2-cell in \O.
  We check \cref{eq:main:15}.
  Let $(A, B)$ be an arbitrary object of $\mb{A} \times \mb{B}$, then
  \[
    \left(\ms{ev}_{\mb{B},\mb{C}} \circ ((\Lambda F) \times \ms{id}_{\mb{B}})\right)(A, B) = \ms{ev}_{\mb{B},\mb{C}}(\Lambda F A, B) = \Lambda F A B = F(A, B).
  \]
  Let $(a, b) : (A, B) \to (A', B')$ be an arbitrary morphism.
  Then
  \begin{align*}
    &\left(\ms{ev}_{\mb{B},\mb{C}} \circ ((\Lambda F) \times \ms{id}_{\mb{B}})\right)(a, b)\\
    &= \ms{ev}_{\mb{B},\mb{C}}(\Lambda F a, b)\\
    &= \Lambda F a \ast b\\
    &= (\Lambda F a)_{B'} \circ (\Lambda F A b)\\
    &= F(a, \ms{id}_{B'}) \circ F(\ms{id}_A, b)\\
    &= F(a, b).
  \end{align*}
  Both sides of \cref{eq:main:15} are equal on objects and morphisms, so define equal functors.
  We conclude \cref{eq:main:15}.
  \Cref{eq:main:16} follows identically.

  We check \cref{eq:main:17}.
  Let $(A, B)$ be an arbitrary object of $\mb{A} \times \mb{B}$.
  We must show that $\left(\ms{ev}_{\mb{B},\mb{C}} \circ ((\Lambda \alpha) \times \ms{id}_{\mb{B}})\right)_{(A,B)} = \alpha_{(A, B)}$.
  We compute:
  \begin{align*}
    &\left(\ms{ev}_{\mb{B},\mb{C}} \circ ((\Lambda \alpha) \times \ms{id}_{\mb{B}})\right)_{(A,B)}\\
    &= \ms{ev}_{\mb{B},\mb{C}}((\Lambda \alpha)_A, \ms{id}_B)\\
    &= (\Lambda \alpha)_A \ast \ms{id}_B\\
    &= \left((\Lambda \alpha)_A\right)_B \circ (\Lambda A \ms{id}_B)\\
    &= \left((\Lambda \alpha)_A\right)_B\\
    &= \alpha_{(A,B)}.
  \end{align*}
  We conclude \cref{eq:main:17}.

  Uniqueness of the 2-cell $\Lambda\alpha$ is inherited from $\mb{Cat}$.
  Indeed, suppose there were some other 2-cell $\beta : F \nto G : \mb{A} \times \mb{B} \to \mb{C}$ making \cref{eq:main:21} commute.
  Observe that {\O} is a subcategory of $\mb{Cat}$, $\Lambda \alpha$ defines the same 2-cell in {\O} as it does in $\mb{Cat}$, and $\ms{ev}_{\mb{B},\mb{C}}$ in {\O} is a restriction of its counterpart in $\mb{Cat}$.
  So $\beta$ also makes \cref{eq:main:21} commute in $\mb{Cat}$.
  We use the fact that $\mb{Cat}$ is a cartesian closed 2-category to conclude $\beta = \Lambda \alpha$.
  This completes the proof.\qed
\end{proof}

\begin{proposition}
  \label{prop:main:6}
  The category {\O} is a cartesian closed 2-category.
\end{proposition}

\begin{proof}
  The category {\O} inherits its 2-categorical structure from $\mb{Cat}$.
  Objects are small {\O} categories, horizontal morphisms are locally continuous functors, and vertical morphisms are natural transformations between locally continuous functors.
  {\O} is a cartesian 2-category by \cref{prop:main:7}.
  Whenever $\mb{A}$ and $\mb{B}$ are small \O-categories, $[\mb{A} \lcto \mb{B}]$ is a small \O-category by \cref{prop:main:11}.
  The evaluation functor $\ms{ev}_{\mb{A},\mb{B}} : [\mb{A} \lcto \mb{B}] \times \mb{A} \to \mb{B}$ is locally continuous by \cref{lemma:main:7}.
  The locally continuous abstraction functor $\Lambda : [\mb{A} \times \mb{B} \lcto \mb{C}] \to [\mb{A} \lcto [\mb{B} \lcto \mb{C}]]$ of has the requisite structure by \cref{lemma:main:8}.
  We conclude that {\O} is cartesian closed 2-category.\qed
\end{proof}

\subsection{Horizontal Composition and Iterates}
\label{sec:horizontal-iterates}

\begin{proposition}
  \label{prop:42}
  Let $\mb{K}_1$, $\mb{K}_2$, and $\mb{K}_3$ be $\mb{O}$-categories.
  Let $F_i, G_i : \mb{K}_{i+1} \to \mb{K}_{i}$ be locally continuous functors and $e_i : F_i \to G_i$ natural embeddings for $i = 1, 2$.
  Then $e_1 \ast e_2 : F_1F_2 \to G_1G_2$ is again a natural embedding with associated projection $e_1^p \ast e_2^p$.
\end{proposition}

\begin{proof}
  This is an immediate consequence of \cref{lemma:main:7,lemma:15}.\qed
\end{proof}

\begin{lemma}
  \label{lemma:main:1}
  Let $\mb{K}$ be an \O-category.
  For all $n \in \N$, the iteration functor $({-})^n : [\mb{K} \lcto \mb{K}] \to [\mb{K} \lcto \mb{K}]$ given by $F \mapsto F^n$ on objects and $\eta \mapsto \eta^{(n)}$ on morphisms is locally continuous.
  It preserves the local continuity of functors.
  Horizontal iteration of natural transformations between locally continuous functors is continuous.
\end{lemma}

\begin{proof}
  The result follows by induction on $n$, \cref{lemma:main:6}, and the fact that locally continuous functors are closed under composition.\qed
\end{proof}

\begin{lemma}
  \label{lemma:main:4}
  Let $\mb{K}$ be an \O-category and $F, G, H : \mb{K} \to \mb{K}$ functors.
  Let $\eta : F \nto G$ and $\rho : G \nto H$ be natural transformations.
  Then for all $n \geq 0$, $(\rho \circ \eta)^{(n)} = \rho^{(n)} \circ \eta^{(n)} : F^n \nto H^n$.
\end{lemma}

\begin{proof}
  Let $K$ be an object of $\mb{K}$.
  We proceed by induction on $n$ to show that $(\rho \circ \eta)^{(n)}_K = (\rho^{(n)} \circ \eta^{(n)})_K : F^nK \to H^nK$.
  When $n = 0$, $F^0K = K = H^0K$ and $(\rho \circ \eta)^{(0)}_K = \ms{id}_K = (\rho^{(0)} \circ \eta^{(0)})_K$.
  Assume the result for some $n$, and consider the case $n + 1$.
  By the induction hypothesis and the middle four interchange law \cites[Lemma~1.7.7]{riehl_2016:_categ_theor_contex}[43]{maclane_1998:_categ_workin_mathem},
  \begin{align*}
    &(\rho \circ \eta)^{(n + 1)}\\
    &= (\rho \circ \eta) \ast (\rho \circ \eta)^{(n)}\\
    &= (\rho \circ \eta) \ast \left(\rho^{(n)} \circ \eta^{(n)}\right)\\
    &= \left(\rho \ast \rho^{(n)}\right) \circ \left(\eta \ast \eta^{(n)}\right)\\
    &= \rho^{(n + 1)} \circ \eta^{(n + 1)}.
  \end{align*}
  We conclude the result by induction.\qed
\end{proof}

\section{Proofs for \cref{sec:background-notation}}
\label{sec:proofs-background-notation}

\propDD*

\begin{proof}
  The proof of the first property is as in \cite[Proposition~A]{smyth_plotkin_1982:_categ_theor_solut}.
  Assume that $\beta : \Phi \nto B$ is a cocone in $\mb{K}^e$.
  We show that $(\alpha_n \circ \beta_n^p)_{n \in \N}$ is an ascending chain in $\mb{K}(A, B)$.
  Recall that, by definition of cocone, $\beta_n = \beta_m \circ \Phi(n \to m)$ for all $m \geq n$.
  For all $n \in \N$, we then have:
  \begin{align*}
    &\alpha_n \circ \beta_n^p\\
    &= (\alpha_{n + 1} \circ \Phi(n \to n + 1)) \circ (\beta_{n + 1} \circ \Phi(n \to n + 1))^p\\
    &= \alpha_{n + 1} \circ (\Phi(n \to n + 1) \circ \Phi(n \to n + 1)^p) \circ \beta_{n + 1}^p\\
    &\sqsubseteq \alpha_{n + 1} \circ \beta_{n + 1}^p.
  \end{align*}
  So its directed supremum $\theta$ exists.

  Set $\theta = \dirsup_{n \in \N} \alpha_n \circ \beta_n^p$.
  We show that $\theta : (\beta, B) \to (\alpha, A)$ is a morphism in $\cel{\Cone_{\mb{K}}(\Phi, {-})}$.
  This means that for all $n$, $\alpha_n = \theta \circ \beta_n$.
  Observe that for all $n$,
  \begin{align*}
    &\theta \circ \beta_n\\
    &= \left(\dirsup_{m \in \N} \alpha_m \circ \beta_m^p\right) \circ \beta_n\\
    &= \dirsup_{m \geq n} \alpha_m \circ \beta_m^p \circ \beta_n\\
    &= \dirsup_{m \geq n} \alpha_m \circ \beta_m^p \circ \beta_m \circ \Phi(n \to m)\\
    &= \dirsup_{m \geq n} \alpha_m \circ \Phi(n \to m)\\
    &= \dirsup_{m \geq n} \alpha_n\\
    &= \alpha_n.
  \end{align*}
  This completes the proof of the first property.

  The proof of the second property is as in \cite[Proposition~A]{smyth_plotkin_1982:_categ_theor_solut}.
  Assume that $\beta : \Phi \nto B$ is an \O-colimit.
  We must show that $\theta$ is an embedding, \ie, that $\theta^p \circ \theta = \ms{id}_B$ and $\theta \circ \theta^p \sqsubseteq \ms{id}_A$.
  In the first case, we use the assumption that $\beta$ is an \O-colimit to get:
  \begin{align*}
    &\theta^p \circ \theta\\
    &= \left(\dirsup_{n \in \N} \alpha_n \circ \beta_n^p\right)^p \circ \left(\dirsup_{n \in \N} \alpha_n \circ \beta_n^p\right)\\
    &= \dirsup_{n \in \N} \left(\alpha_n \circ \beta_n^p\right)^p \circ \left(\dirsup_{n \in \N} \alpha_n \circ \beta_n^p\right)\\
    &= \dirsup_{n \in \N} \left(\alpha_n \circ \beta_n^p\right)^p \circ \alpha_n \circ \beta_n^p\\
    &= \dirsup_{n \in \N} \beta_n \circ \alpha_n^p \circ \alpha_n \circ \beta_n^p\\
    &= \dirsup_{n \in \N} \beta_n \circ \beta_n^p\\
    &= \ms{id}_B.
  \end{align*}
  In the second case,
  \begin{align*}
    &\theta \circ \theta^p\\
    &= \left(\dirsup_{n \in \N} \alpha_n \circ \beta_n^p\right) \circ \left(\dirsup_{n \in \N} \alpha_n \circ \beta_n^p\right)^p\\
    &= \left(\dirsup_{n \in \N} \alpha_n \circ \beta_n^p\right) \circ \dirsup_{n \in \N} \left(\alpha_n \circ \beta_n^p\right)^p\\
    &= \dirsup_{n \in \N} \alpha_n \circ \beta_n^p \circ \left(\alpha_n \circ \beta_n^p\right)^p\\
    &= \dirsup_{n \in \N} \alpha_n \circ \beta_n^p \circ \beta_n \circ \alpha_n^p\\
    &= \dirsup_{n \in \N} \alpha_n \circ \alpha_n^p\\
    &\sqsubseteq \ms{id}_A.
  \end{align*}
  This completes the proof of the second property.

  The proof of the third property is as in \cite[Proposition~A]{smyth_plotkin_1982:_categ_theor_solut}.
  Assume that $\alpha$ is an \O-colimit.
  By definition, this means $\alpha$ lies in $\mb{K}^e$.
  We show that $\alpha$ is colimiting in $\mb{K}$ and $\mb{K}^e$.
  Consider some other cocone $\delta : \Phi \nto D$ in $\mb{K}$.
  There exists a cocone morphism $\theta = \dirsup_{n \in \N} \delta_n \circ \alpha_n^p : (\alpha, A) \to (\delta, D)$ by the first property.
  We show that it is unique.
  Let $\gamma : (\alpha, A) \to (\delta, D)$ be any other cocone morphism.
  Then
  \[
    \gamma = \gamma \circ \ms{id}_A = \gamma \circ \left(\dirsup_{n \in \N} \alpha_n \circ \alpha_n^p\right) = \dirsup_{n \in \N} (\gamma \circ \alpha_n) \circ \alpha_n^p = \dirsup_{n \in \N} \delta_n \circ \alpha_n^p = \theta.
  \]
  We conclude that $\alpha$ is colimiting in $\mb{K}$.
  The category $\mb{K}^e$ is a subcategory of $\mb{K}$ and the mediating morphism $\theta$ is an embedding by the second property.
  It follows that $\alpha$ is also colimiting in $\mb{K}^e$
  This completes the proof of the third property.

  The fourth property is exactly \cite[Proposition~D]{smyth_plotkin_1982:_categ_theor_solut} and is not reproduced here.\qed
\end{proof}

\section{Proofs for \cref{sec:funct-fixed-points}}
\label{sec:proofs-funct-fixed-points}

\propMainB*

\begin{proof}
  In this proof we show that:
  \begin{enumerate}
  \item $\Omega\lk{K}{k}{F}$ is a well-defined functor $\omega \to \mb{K}^e$ for all links $\lk{K}{k}{F}$;
  \item $\Omega(f, \eta)$ is natural;
  \item $\Omega$ respects composition;
  \item $\Omega$ is locally continuous;
  \item $\Omega(f, \eta)$ lies in $\mb{K}^e$ whenever $f$ and $\eta$ do.
  \end{enumerate}
  Let $\lk{K}{k}{F}$ be an arbitrary link and abbreviate $\Omega\lk{K}{k}{F}$ by $J$.
  We must show that $J$ is a well-defined functor $\omega \to \mb{K}^e$.
  It preserves identities by \cref{eq:main:4}.
  We must show that it respects composition.
  Let $l \to l + m$ and $l + m \to l + m + n$ be arbitrary, and note that $l \to l + m + n = (l + m \to l + m + n) \circ (l \to l + m)$.
  We must show that $J(l \to l + m + n) = J(l + m \to l + m + n) \circ J(l \to l + m)$.
  We proceed by nested strong induction on $n$.
  Assume first $n = 0$, then by \cref{eq:main:4},
  \begin{align*}
    &J(l \to l + m + n)\\
    &= J(l \to l + m)\\
    &= \ms{id}_{F^{l + m}K} \circ J(l \to l + m)\\
    &= J(l + m \to l + m) \circ J(l \to l + m)\\
    &= J(l + m \to l + m + n) \circ J(l \to l + m).
  \end{align*}
  Now assume the result for some $n$, then by \cref{eq:main:2},
  \begin{align*}
    &J(l \to l + m + (n + 1))\\
    &= F^{l + m + n}k \circ J(l \to l + m + n)\\
    &= F^{l + m + n}k \circ \left(J(l + m \to l + m + n) \circ J(l \to l + m)\right)\\
    &= (F^{l + m + n}k \circ J(l + m \to l + m + n)) \circ J(l \to l + m)\\
    &= J(l + m \to l + m + (n + 1)) \circ J(l \to l + m).
  \end{align*}
  We conclude the result by induction.

  Next, we must show that $\Omega$ is well-defined on morphisms.
  Let $(f, \eta) : \lk{K}{k}{F} \to \lk{L}{l}{G}$ be arbitrary.
  We must show that $\Omega(f, \eta) : \Omega\lk{K}{k}{F} \nto \Omega\lk{L}{l}{G}$ is a natural transformation.
  Let $n \to n + m$ be an arbitrary morphism of $\omega$.
  We must show that the following diagram commutes:
  \begin{equation}
    \label[diagram]{eq:main:5}
    \begin{tikzcd}[column sep=4em]
      F^nK \ar[r, "{\Omega(f, \eta)_n}"] \ar[d, swap, "{\Omega(K,F,k)(n \to n + m)}"] & G^nL \ar[d, "{\Omega(L,G,l)(n \to n + m)}"]\\
      F^{n + m}K \ar[r, "{\Omega(f, \eta)_{n + m}}"] & G^{n + m}L
    \end{tikzcd}
  \end{equation}
  We proceed by induction on $m$.
  Assume first that $m = 0$, then \cref{eq:main:5} becomes
  \[
    \begin{tikzcd}[column sep=4em]
      F^nK \ar[r, "{\Omega(f, \eta)_n}"] \ar[d, swap, "{\ms{id}}"] & G^nL \ar[d, "{\ms{id}}"]\\
      F^{n}K \ar[r, "{\Omega(f, \eta)_{n}}"] & G^{n}L
    \end{tikzcd}
  \]
  and clearly commutes.
  Assume the result for some $m$, and consider the case $m + 1$.
  We must show that the following diagram commutes:
  \[
    \begin{tikzcd}[column sep=5em]
      F^nK \ar[r, "{\Omega(f, \eta)_n}"] \ar[d, swap, "{\Omega(K,F,k)(n \to n + m + 1)}"] & G^nL \ar[d, "{\Omega(L,G,l)(n \to n + m + 1)}"]\\
      F^{n + m + 1}K \ar[r, "{\Omega(f, \eta)_{n + m + 1}}"] & G^{n + m + 1}L
    \end{tikzcd}
  \]
  By \cref{eq:main:2}, this diagram is equal to the outer rectangle of the following diagram:
  \begin{equation}
    \label[diagram]{eq:main:6}
    \begin{tikzcd}[column sep=5em]
      F^nK \ar[r, "{\Omega(f, \eta)_n}"] \ar[d, swap, "{\Omega(K,F,k)(n \to n + m)}"] & G^nL \ar[d, "{\Omega(L,G,l)(n \to n + m)}"]\\
      F^{n + m}K \ar[r, "{\Omega(f, \eta)_{n + m}}"] \ar[d, swap, "{F^{n + m}k}"] & G^{n + m}L \ar[d, "{G^{n + m}l}"]\\
      F^{n + m + 1}K \ar[r, "{\Omega(f, \eta)_{n + m + 1}}"] & G^{n + m + 1}L
    \end{tikzcd}
  \end{equation}
  The top square commutes by the induction hypothesis.
  The middle horizontal morphism is $\Omega(f, \eta)_{n + m} = \eta^{(n + m)} \ast f = G^{n + m}f \circ \eta^{(n + m)}_K$.
  Horizontal composition is associative, and the bottom morphism is $\Omega(f, \eta)_{n + m + 1} = \eta^{(n + m + 1)} \ast f = \eta \ast \eta^{(n + m)} \ast f = \eta^{(n + m)} \ast \eta \ast f$.
  The bottom square of \cref{eq:main:6} is equal to the outer rectangle of the following diagram:
  \begin{equation}
    \label[diagram]{eq:main:7}
    \begin{tikzcd}[column sep=5em]
      F^{n + m}K \ar[r, "{\eta^{(n + m)}_K}"] \ar[d, swap, "{F^{n + m}k}"] & G^{n + m}K \ar[r, "{G^{n + m}f}"] \ar[d, "{G^{n + m}k}"] & G^{n + m}L \ar[d, "{G^{n + m}l}"]\\
      F^{n + m + 1}K \ar[r, "{\eta^{(n + m)}_{FK}}"] & G^{n + m}FK \ar[r, "{G^{n + m}(\eta \ast f)}"] & G^{n + m + 1}L
    \end{tikzcd}
  \end{equation}
  The left square of \cref{eq:main:7} commutes by naturality of $\eta^{(n + m)}$.
  The right square commutes because $l \circ f = (\eta \ast f) \circ k$, which holds because $(f, \eta)$ is a morphism.
  So \cref{eq:main:7} commutes.
  We conclude that \cref{eq:main:6} commutes.

  Next, we must show that $\Omega$ respects composition.
  Let $(f, \eta) : \lk{K}{k}{F} \to \lk{L}{l}{G}$ and $(g, \rho) : \lk{L}{l}{G} \to \lk{M}{m}{H}$ be arbitrary morphisms.
  We must show that $\Omega(g \circ f, \rho \circ \eta) = \Omega(g, \rho) \circ \Omega(f, \eta)$.
  This entails showing for all $n \in \N$ that $\Omega(g \circ f, \rho \circ \eta)_n = \Omega(g, \rho)_n \circ \Omega(f, \eta)_n$.
  We proceed by induction on $n$.
  When $n = 0$, $\Omega(g \circ f, \rho \circ \eta)_0 = g \circ f = \Omega(g, \rho)_0 \circ \Omega(f, \eta)_0$.
  Assume the result for some $n$, and consider the case $n + 1$.
  Then $\Omega(g \circ f, \rho \circ \eta)_{n + 1} = (\rho \circ \eta)^{n + 1} \ast (g \circ f)$ is the diagonal of the following commuting square:
  \[
    \begin{tikzcd}[column sep=6em]
      F^{n + 1}K \ar[r, "{(\rho \circ \eta)^{(n + 1)}_K}"] \ar[d, swap, "{F^{n + 1}(g \circ f)}"]
      \ar[dr, dashed, "{\Omega(g \circ f, \rho \circ \eta)_{n + 1}}" description]
      & H^{n + 1}K \ar[d, "{H^{n + 1}(g \circ f)}"]\\
      F^{n + 1}M \ar[r, swap, "{(\rho \circ \eta)^{(n + 1)}_M}"] & H^{n + 1}M.
    \end{tikzcd}
  \]
  By \cref{lemma:main:4}, $(\rho \circ \eta)^{(n + 1)} = \rho^{(n + 1)} \circ \eta^{(n + 1)}$.
  We recognize the above diagram as the perimeter and diagonal of the following commuting diagram:
  \[
    \begin{tikzcd}[column sep=6em]
      F^{n + 1}K \ar[r, "{\eta^{(n + 1)}_K}"] \ar[d, swap, "{F^{n + 1}f}"] \ar[dr, dashed, "{\Omega(f, \eta)_{n + 1}}" description]
      & G^{n + 1}K \ar[r, "{\rho^{(n + 1)}_K}"] \ar[dr, dashed, "{\Omega(f, \rho)_{n + 1}}" description] \ar[d, "{G^{n + 1}f}" description]
      & H^{n + 1}K \ar[d, "{H^{n + 1}f}"]\\
      F^{n + 1}L \ar[r, swap, "{\eta^{(n + 1)}_L}" description]  \ar[d, swap, "{F^{n + 1}g}"] \ar[dr, dashed, "{\Omega(g, \eta)_{n + 1}}" description]
      & G^{n + 1}L \ar[r, "{\rho^{(n + 1)}_L}" description] \ar[dr, dashed, "{\Omega(g, \rho)_{n + 1}}" description] \ar[d, "{G^{n + 1}g}" description]
      & H^{n + 1}L \ar[d, "{H^{n + 1}f}"]\\
      F^{n + 1}M \ar[r, swap, "{\eta^{(n + 1)}_M}"]
      & G^{n + 1}M \ar[r, swap, "{\rho^{(n + 1)}_M}"]
      & H^{n + 1}M
    \end{tikzcd}
  \]
  That is, we have $\Omega(g \circ f, \rho \circ \eta)_{n + 1} = (\Omega(g, \rho) \circ \Omega(f, \eta))_{n + 1}$.
  We conclude by induction that $\Omega(g \circ f, \rho \circ \eta) = \Omega(g, \rho) \circ \Omega(f, \eta)$.

  We show that $\Omega$ is locally continuous.
  Now consider a directed set $A$ of morphisms $\lk{K}{k}{F} \to \lk{L}{l}{G}$.
  We must show that $\dirsup_{(f, \eta) \in A} \Omega(f, \eta) = \Omega\left(\dirsup_{(f, \eta) \in A} (f, \eta)\right)$.
  We have for all $n \in \N$:
  \begin{align*}
    &\left(\dirsup_{(f, \eta) \in A} \Omega(f, \eta)\right)_n\\
    &= \dirsup_{(f, \eta) \in A} \Omega(f, \eta)_n\\
    &= \dirsup_{(f, \eta) \in A} \eta^{(n)} \ast f\\
    &= \dirsup_{(f, \eta) \in A} G^nf \circ \eta^{(n)}_K\\
    &= \left(\dirsup_{(f, \eta) \in A} G^nf\right) \circ \left(\dirsup_{(f, \eta) \in A} \eta^{(n)}_K\right)\\
    \shortintertext{which by local continuity of $G^n$,}
    &=  G^n\left(\dirsup_{(f, \eta) \in A} f\right) \circ \left(\dirsup_{(f, \eta) \in A} \eta^{(n)}_K\right)\\
    \shortintertext{which by \cref{lemma:main:1},}
    &=  G^n\left(\dirsup_{(f, \eta) \in A} f\right) \circ \left(\dirsup_{(f, \eta) \in A} \eta_K\right)^{(n)}\\
    &=  G^n\left(\dirsup_{(f, \eta) \in A} f\right) \circ \left(\dirsup_{(f, \eta) \in A} \eta\right)_K^{(n)}\\
    &= \left(\dirsup_{(f, \eta) \in A} \eta\right) \ast \left(\dirsup_{(f, \eta) \in A} f\right)\\
    &= \Omega\left(\dirsup_{(f, \eta) \in A} (f, \eta)\right).
  \end{align*}
  We conclude local continuity.

  Finally, let $(f, \eta) : \lk{K}{k}{F} \to \lk{L}{l}{G}$ be arbitrary, and assume $f$ and $\eta$ lie in $\mb{K}^e$.
  We show that $\Omega(f, \eta)$ lies in $\mb{K}^e$.
  To do so, we must show that $\Omega(f, \eta)_n$ lies in $\mb{K}^e$ for all $n \in \N$.
  Let $n \in \N$ be arbitrary.
  Then $\Omega(f, \eta)_n = \eta^{(n)} \ast f = G^nf \circ \eta^{(n)}_K$.
  By induction on $n$ and \cref{prop:42}, $\eta^{(n)}_K$ is an embedding.
  Locally continuous functors preserve embeddings by \cref{lemma:15}, so $G^nf$ is an embedding.
  Embeddings are closed under composition, so we conclude that $\Omega(f, \eta)_n$ is an embedding.
  \qed
\end{proof}

\Cref{prop:main:10} is a corollary of \cref{prop:33}.

\propMainBA*

\begin{proof}
  The first part of this corollary is a special case of \cite[Proposition~3.6.1]{riehl_2016:_categ_theor_contex}.
  Consider two $\omega$-chains $\Phi$ and $\Gamma$ in $\mb{K}^e$.
  Let $\phi : \Phi \nto \colim_\omega \Phi$ and $\gamma : \Gamma \nto \colim_\omega \Gamma$ be the chosen \O-colimits.
  Then $\gamma \circ \eta : \Phi \nto \colim_\omega \Gamma$ is again a cocone by naturality of $\eta$.
  By \cref{prop:33}, the unique mediating morphism from $\phi$ to $\gamma \circ \eta$ is exactly $\colim_\omega\eta = \dirsup_{n \in \N} \gamma_n \circ \eta_n \circ \phi_n^p$.
  Because $\phi$ is an \O-colimit, $\colim_\omega\eta$ is an embedding, again by \cref{prop:33}.
  This mapping on morphisms is functorial by uniqueness of mediating morphisms.

  Next, we show that $\colim_\omega$ is locally continuous.
  Let $\Phi$ and $\Gamma$ be arbitrary $\omega$-chains in $\mb{K}^e$ and let $\phi$ and $\gamma$ be their respective chosen \O-colimits.
  By the above, both are \O-colimits.
  Let $A$ by a directed set of natural transformations from $\Phi$ to $\Gamma$ in $\mb{K}$ (or in $\mb{K}^e$ if $\mb{K}$ has locally determined $\omega$-colimits of embeddings).
  We must show that $\colim_\omega \left(\dirsup A\right) = \dirsup \left(\colim_\omega A\right)$.
  By the above, we have
  \begin{align*}
    &\colim_\omega \left(\dirsup A\right)\\
    &= \dirsup_{n \in \N} \gamma_n \circ \left(\dirsup A\right)_n \circ \phi_n^p\\
    &= \dirsup_{n \in \N} \gamma_n \circ \left(\dirsup_{\alpha \in A} \alpha_n\right) \circ \phi_n^p\\
    &= \dirsup_{n \in \N} \dirsup_{\alpha \in A} \gamma_n \circ \alpha_n \circ \phi_n^p \\
    &= \dirsup_{\alpha \in A} \dirsup_{n \in \N} \gamma_n \circ \alpha_n \circ \phi_n^p \\
    &= \dirsup_{\alpha \in A} \left(\colim_\omega \alpha\right)\\
    &= \dirsup \left(\colim_\omega A\right).
  \end{align*}
  We conclude that $\colim_\omega$ is locally continuous.\qed
\end{proof}

\propMainI*

\begin{proof}
  The composition is clearly well defined.
  Assume $\mb{K}$ has locally determined $\omega$-colimits of embeddings.
  The functor $\Omega$ is locally continuous by \cref{prop:main:1}.
  The $\omega$-colimit functor is locally continuous by \cref{prop:main:10}.
  Locally continuous functors are closed under composition, so we conclude $\GFIX$ is locally continuous.
  The action of $\GFIX$ on morphisms is given by \cref{prop:main:10}.
\end{proof}

\propMainC*

\begin{proof}
  We begin by showing that $\UNF$ is well-defined on morphisms.
  Let $(f, \eta) : \lk{K}{k}{F} \to \lk{L}{l}{G}$ be arbitrary.
  We show that $\UNF(f, \eta) : \UNF\lk{K}{k}{F} \to \UNF\lk{L}{l}{G}$ is a morphism of $\mb{K}^e$.
  Let $\phi : \Omega\lk{K}{k}{F} \nto \GFIX\lk{K}{k}{F}$ and $\gamma : \Omega\lk{L}{l}{G} \nto \GFIX\lk{L}{l}{G}$ be the chosen \O-colimits in $\mb{K}$.
  Observe that $F\Omega\lk{K}{k}{F}(n) = \Omega\lk{K}{k}{F}(n + 1)$ and $F\Omega\lk{K}{k}{F}(n \to m) = \Omega\lk{K}{k}{F}(n + 1 \to m + 1)$, and analogously for $G\Omega\lk{L}{l}{G}$.
  Observe that $\Omega(f, \eta) : \Omega\lk{K}{k}{F} \nto \Omega\lk{L}{l}{G}$ is a natural transformation and that reindexing gives a natural transformation $(\Omega(f, \eta)_{n + 1})_n : F\Omega\lk{K}{k}{F} \nto G\Omega\lk{L}{l}{G}$.
  Note that $G\gamma$ is a cocone on $G\Omega\lk{L}{l}{G}$.
  It follows that $G\gamma \circ (\Omega(f, \eta)_{n + 1})_n : F\Omega\lk{K}{k}{F} \nto G(\GFIX\lk{L}{l}{G})$ is also a cocone on $F\Omega\lk{K}{k}{F}$.
  But $F\phi$ is also a cocone on $F\Omega\lk{K}{k}{F}$.
  By \cref{prop:33}, it follows that
  \begin{equation}
    \label{eq:main:8}
    \dirsup_{n \in \N} \left(G\gamma \circ (\Omega(f, \eta)_{m + 1})_m\right)_n \circ (F\phi)^p_n
  \end{equation}
  is a morphism $(F\phi, F\GFIX\lk{K}{k}{F}) \to (G\gamma \circ (\Omega(f, \eta)_{n + 1})_n, G(\GFIX\lk{L}{l}{G}))$ in $\cel{\Cone(F\Omega\lk{K}{k}{F}, {-})}$.
  We observe that \cref{eq:main:8} is $\UNF(f, \eta)$:
  \begin{align*}
    &\dirsup_{n \in \N} \left(G\gamma \circ (\Omega(f, \eta)_{m + 1})_m\right)_n \circ (F\phi)^p_n\\
    &= \dirsup_{n \in \N} G\gamma_n \circ \Omega(f, \eta)_{n + 1} \circ F\phi^p_n\\
    &= \UNF(f, \eta).
  \end{align*}
  Locally continuous functors preserve \O-colimits, so $F\phi$ is again an \O-colimit.
  In this case, \cref{eq:main:8} is an embedding by \cref{prop:33}.
  It follows that $\UNF : \Links_{\mb{K}} \to \mb{K}^e$ whenever $\mb{K}$ has locally determined $\omega$-colimits of embeddings.

  Next we show that $\UNF$ is locally continuous.
  Let $A$ be a directed set of morphisms $\lk{K}{k}{F} \to \lk{L}{l}{G}$.
  We must show that $\UNF\left(\dirsup A\right) = \dirsup_{(f, \eta) \in A} \UNF(f, \eta)$.
  Recall that $\Omega$ is locally continuous by \cref{prop:main:1}.
  We compute that:
  \begin{align*}
    &\dirsup_{(f, \eta) \in A} \UNF(f, \eta)\\
    &= \dirsup_{(f, \eta) \in A} \dirsup_{n \in \N} G\gamma_n \circ \Omega(f, \eta)_{n + 1} \circ F\phi^p_n\\
    &= \dirsup_{n \in \N} G\gamma_n \circ \left(\dirsup_{(f, \eta) \in A} \Omega(f, \eta)_{n + 1} \right) \circ F\phi^p_n\\
    &= \dirsup_{n \in \N} G\gamma_n \circ \left(\dirsup_{(f, \eta) \in A} \Omega(f, \eta)\right)_{n + 1} \circ F\phi^p_n\\
    &= \dirsup_{n \in \N} G\gamma_n \circ \Omega\left(\dirsup A\right)_{n + 1} \circ F\phi^p_n\\
    &= \UNF\left(\dirsup A\right).
  \end{align*}
  We conclude that $\UNF$ is locally continuous.\qed
\end{proof}

\begin{proposition}[{\Cref{prop:main:3}}]
  \label{prop:main:5}
  Let $\mb{K}$ be an \O-cocomplete \O-category.
  There exists a natural isomorphism $\ms{fold} : \UNF \nto \GFIX$ with inverse $\ms{unfold} : \GFIX \nto \UNF$.
  They are explicitly given as follows.
  Let $\lk{K}{k}{F}$ be an object of $\Links_{\mb{K}}$ and let $\kappa : \Omega\lk{K}{k}{F} \nto \GFIX\lk{K}{k}{F}$ be the chosen \O-colimit.
  The components are:
  \begin{align*}
    \ms{fold}_{\lk{K}{k}{F}} &= \dirsup_{n \in \N} \kappa_{n+1} \circ F\kappa_n^p :  F(\GFIX\lk{K}{k}{F}) \to \GFIX\lk{K}{k}{F}\\
    \ms{unfold}_{\lk{K}{k}{F}} &= \dirsup_{n \in \N} F\kappa_n \circ \kappa_{n+1}^p : \GFIX\lk{K}{k}{F} \to F(\GFIX\lk{K}{k}{F}).
  \end{align*}
  Naturality means that for every $(f, \eta) : \lk{K}{k}{F} \to \lk{L}{l}{G}$, the following diagram commutes:
  \begin{equation}%
    \begin{tikzcd}[column sep=4em]%
      F(\GFIX\lk{K}{k}{F}) \ar[r, shift left=0.5ex, "{\ms{fold}_{\lk{K}{k}{F}}}"] \ar[d, swap, "{\UNF(f, \eta)}"]  & \GFIX\lk{K}{k}{F} \ar[l, shift left=0.5ex, "{\ms{unfold}_{\lk{K}{k}{F}}}"]  \ar[d, "{\GFIX(f, \eta)}"]\\%
      G(\GFIX\lk{L}{l}{G}) \ar[r, shift left=0.5ex, "{\ms{fold}_{\lk{L}{l}{G}}}"] & \GFIX\lk{L}{l}{G} \ar[l, shift left=0.5ex, "{\ms{unfold}_{\lk{L}{l}{G}}}"].%
    \end{tikzcd}%
  \end{equation}
\end{proposition}

\begin{proof}[{of \cref{prop:main:3}}]
  We first show that for each link $\lk K k F$, $\ms{fold}_{\lk K k F}$ is a morphism $F(\GFIX\lk{K}{k}{F}) \to \GFIX\lk{K}{k}{F}$.
  Let $(\kappa, \GFIX\lk{K}{k}{F})$ be the chosen \O-colimit of $\Omega\lk{K}{k}{F}$.
  Locally continuous functors preserve \O-colimits, so $(F\kappa, F(\GFIX\lk{K}{k}{F}))$ is an \O-colimit of $F\Omega\lk{K}{k}{F}$.
  Let $(\kappa^-, \GFIX\lk{K}{k}{F})$ be the cocone on $F\Omega\lk{K}{k}{F}$ given by $\kappa^-_n = \kappa_{n + 1}$.
  By \cref{prop:33}, the cocone morphism $(F\kappa, F(\GFIX\lk{K}{k}{F})) \to (\kappa^-, \GFIX\lk{K}{k}{F})$ is exactly $\ms{fold}_{\lk K k F}$.
  It is an embedding because the cocone $(F\kappa, F(\GFIX\lk{K}{k}{F}))$ is an \O-colimit.

  Next, we show that $\ms{fold}_{\lk K k F}$ is an isomorphism with inverse $\ms{unfold}_{\lk K k F}$.
  Observe that
  \[
    \ms{fold}_{\lk K k F}^p = \left(\dirsup_{n \in \N} \kappa_{n+1} \circ F\kappa_n^p\right)^p = \dirsup_{n \in \N} F\kappa_n \circ \kappa_{n+1}^p = \ms{unfold}_{\lk{K}{k}{F}}.
  \]
  We already know that $\ms{unfold}_{\lk{K}{k}{F}} \circ \ms{fold}_{\lk{K}{k}{F}} = \ms{id}$ because $\ms{fold}_{\lk{K}{k}{F}}$ is an embedding.
  Using the fact that $\kappa$ is an $\O$-colimit, we compute:
  \begin{align*}
    &\ms{fold}_{\lk{K}{k}{F}} \circ \ms{unfold}_{\lk{K}{k}{F}}\\
    &= \left(\dirsup_{n \in \N} \kappa_{n+1} \circ F\kappa_n^p\right) \circ \left(\dirsup_{n \in \N} F(\kappa_n) \circ \kappa_{n+1}^p\right)\\
    &= \dirsup_{n \in \N} \kappa_{n+1} \circ F\kappa_n^p \circ F\kappa_n \circ \kappa_{n+1}^p\\
    &= \dirsup_{n \in \N} \kappa_{n+1} \circ \kappa_{n+1}^p\\
    &= \ms{id}
  \end{align*}
  We conclude that $\ms{fold}_{\lk K k F}$ is an isomorphism with inverse $\ms{unfold}_{\lk K k F}$.

  Finally we show naturality.
  Let $(f, \eta) : \lk{K}{k}{F} \to \lk{L}{l}{G}$ be an arbitrary morphism of $\Links_{\mb{k}}$.
  We must show that \cref{eq:main:9} commutes.
  Because $\ms{unfold}$ is an isomorphism with inverse $\ms{fold}$, to show naturality of $\ms{fold}$ and $\ms{unfold}$ it is sufficient by~\cite[Lemma~1.5.10]{riehl_2016:_categ_theor_contex} to show that $\ms{fold}$ is natural, \ie, that the following square commutes:
  \[
    \begin{tikzcd}[column sep=4em]
      F(\GFIX\lk{K}{k}{F}) \ar[r, shift left=0.5ex, "{\ms{fold}_{\lk{K}{k}{F}}}"] \ar[d, swap, "{\UNF(f, \eta)}"]  & \GFIX\lk{K}{k}{F} \ar[d, "{\GFIX(f, \eta)}"]\\
      G(\GFIX\lk{L}{l}{G}) \ar[r, shift left=0.5ex, "{\ms{fold}_{\lk{L}{l}{G}}}"] & \GFIX\lk{L}{l}{G}.
    \end{tikzcd}
  \]
  Let $(\phi, \GFIX\lk{K}{k}{F})$ and $(\gamma, \GFIX\lk{L}{l}{G})$ respectively be the colimiting cocones of $\Omega\lk{K}{k}{F}$ and $\Omega\lk{L}{l}{G}$.
  We compute:
  \begin{align*}
    &\GFIX(f, \eta) \circ \ms{fold}_{\lk{K}{k}{F}}\\
    &= \left(\dirsup_{n \in \N} \gamma_n \circ \Omega(f, \eta)_n \circ \phi^p_n\right) \circ \left(\dirsup_{n \in \N} \phi_{n + 1} \circ F\phi_n^p\right)\\
    &= \left(\dirsup_{n \in \N} \gamma_{n + 1} \circ \Omega(f, \eta)_{n + 1} \circ \phi^p_{n + 1}\right) \circ \left(\dirsup_{n \in \N} \phi_{n + 1} \circ F\phi_n^p\right)\\
    &= \dirsup_{n \in \N} \gamma_{n + 1} \circ \Omega(f, \eta)_{n + 1} \circ \phi^p_{n + 1} \circ \phi_{n + 1} \circ F\phi_n^p\\
    &= \dirsup_{n \in \N} \gamma_{n + 1} \circ \Omega(f, \eta)_{n + 1} \circ F\phi_n^p\\
    &= \dirsup_{n \in \N} \gamma_{n + 1} \circ G\gamma_n^p \circ G\gamma_n \circ \Omega(f,\eta)_{n + 1} \circ F\phi_n^p\\
    &= \left(\dirsup_{n \in \N} \gamma_{n + 1} \circ G\gamma_n^p\right) \circ \left(\dirsup_{n \in \N} G\gamma_n \circ \Omega(f,\eta)_{n + 1} \circ F\phi_n^p\right)\\
    &= \ms{fold}_{\lk{L}{l}{G}} \circ \UNF(f, \eta).
  \end{align*}
  We conclude that $\ms{fold}$ is a natural transformation.\qed
\end{proof}

\begin{proposition}
  \label{prop:main:9}
  Let $\mb{K}$ be an \O-category and assume $\mb{K}^e$ has an initial object.
  Let $I : [\mb{K} \lcto \mb{K}] \to \Links_{\mb{K}}$ be the functor given by $I(F) = \lk{\bot}{\bot}{F}$ and $I(\eta) = (\ms{id}_\bot, \eta)$.
  Then $I$ is locally continuous and full and faithful.
\end{proposition}

\begin{proof}
  The mapping $I$ is clearly functorial and locally continuous.
  Let $\eta, \rho : F \nto G$ be two morphisms in $[\mb{K} \lcto \mb{K}]$ and assume $I(\eta) = I(\rho)$.
  Then $(\ms{id}_\bot, \eta) = (\ms{id}_\bot, \rho)$.
  It follows that $\eta = \rho$, so $I$ is faithful.

  Let $F, G : \mb{K} \lcto \mb{K}$ be two locally continuous functors.
  Let $(f, \eta) : I(F) \to I(G)$ be a morphism in $\Links_{\mb{K}}$.
  Then $(f, \eta) : \lk{\bot}{\bot}{F} \to \lk{\bot}{\bot}{G}$.
  This implies that $f : \bot \to \bot$.
  But $\bot$s is the initial object, and there exists a unique morphism $\bot \to \bot$, namely, $\ms{id}_\bot$.
  It follows that $(f, \eta) = I(\eta)$.
  We conclude that $I$ is full.
\end{proof}

The definition of strict morphisms is motivated in part by \cref{prop:44}.
Given a subcategory $\mb{P}$ of $\mb{Poset}$, we write $\strc{\mb{P}}$ for the subcategory of $\mb{P}$ whose objects are pointed posets and whose morphisms are bottom-preserving.

\begin{proposition}
  \label{prop:44}
  Let $\mb{P}$ be a subcategory of $\mb{Poset}$.
  If $\mb{P}$ is an $\mb{O}$-category with strict morphisms and its initial object is a singleton poset $\{\ast\}$, then $\mb{P}$ is a subcategory of $\strc{\mb{Poset}}$.
\end{proposition}

\begin{proof}
  To avoid confusion, we write $\iota$ for the limiting cone $\bot$.

  We begin by showing that the objects of $\mb{P}$ are pointed posets.
  Let $Q$ be any object of $\mb{P}$.
  We know that $Q$ is non-empty because $\iota_Q(\ast) \in Q$ by initiality.
  Write $\bot_Q$ for the image $\iota_Q(\ast)$.
  We claim that for all elements $q \in Q$, $\bot_Q \sqsubseteq q$.
  Indeed, by definition of e-p-pair, $\iota_Q \circ \iota_Q^p \sqsubseteq \ms{id}_Q$, so
  \[
    \bot_Q = \iota_Q(\ast) = \iota_Q(\iota_Q^pq) = (\iota_Q \circ \iota_Q^p)(q) \sqsubseteq \ms{id}_Q(q) = q.
  \]
  We conclude that $Q$ is a pointed partial order.

  Next, we show that all morphisms are bottom-preserving.
  Let $f : P \to Q$ be arbitrary.
  By initiality, $f \circ \iota_P = \iota_Q$, so $(f \circ \iota_P)(\ast) = \iota_Q(\ast)$, \ie, $f(\bot_P) = \bot_Q$.

  We conclude that $\mb{P}$ is a subcategory of $\strc{\mb{Poset}}$.\qed
\end{proof}

\begin{proposition}
  \label{prop:main:14}
  If $\mb{K}$ is an \O-category with strict morphisms and an initial object $\bot : \bot_{\mb{K}} \nto \ms{id}_{\mb{K}}$, then $\bot$ is also initial in $\mb{K}^e$.
\end{proposition}

\begin{proof}
  We must show that $\bot_A : \bot_{\mb{K}} \to A$ is an embedding for all $A$.
  We claim that $\bot_A^p = 0_{A\bot_{\mb{K}}}$.
  We have by initiality
  \[
    \bot_{\mb{K}} \xrightarrow{\bot_A} A \xrightarrow{0_{A\bot_{\mb{K}}}} \bot_{\mb{K}} = \bot_{\mb{K}} \xrightarrow{\ms{id}} \bot_{\mb{K}}.
  \]
  By definition of zero morphism, $\bot_A \circ \bot_A^p = \bot_A \circ 0_{A\bot_{\mb{K}}} = 0_{AA}$.
  By definition of strict morphism, $0_{AA} \sqsubseteq \ms{id}_A$.
  So $\bot_A \circ \bot_A^p \sqsubseteq \ms{id}_A$ by transitivity.
  We conclude that $\bot_A$ is an embedding.
  Because $A$ was arbitrary, $\bot$ lies in $\mb{K}^e$.
  $\mb{K}^e$ is a subcategory of $\mb{K}$, so uniqueness of $\bot$ is inherited.
  We conclude that $\bot$ is initial in $\mb{K}^e$.\qed
\end{proof}

\PropMainE*

\begin{proof}
  The functor $\FIX$ is the composition of locally continuous functors
  \[
    \left[\mb{D} \times \mb{E} \lcto \mb{E}\right] \xrightarrow{\Lambda} \left[\mb{D} \lcto \left[ \mb{E} \lcto \mb{E}\right]\right] \xrightarrow{F \mapsto FIX \circ F} \left[\mb{D} \lcto \mb{E} \right].
  \]
  Locally continuous functors are closed under composition.
  Let $F, G : \mb{D} \times \mb{E} \to \mb{E}$ be locally continuous and $\eta : F \nto G$ a natural transformation.
  We have
  \[
    \sfix{F} = \left([\ms{id}_{\mb{D}} \to \FIX] \circ \Lambda\right)F = \FIX \circ \Lambda F \circ \ms{id}_{\mb{D}} = \FIX \circ \Lambda F,
  \]
  so
  \[
    \sfix{F}D = (\FIX \circ \Lambda F)D = \FIX(\Lambda F D) = \FIX(F_D).
  \]
  Analogously,
  \[
    \sfix{\eta} = \left([\ms{id}_{\mb{D}} \to \FIX] \circ \Lambda\right)\eta = \FIX (\Lambda \eta) \ms{id}_{\mb{D}} = \FIX (\Lambda \eta),
  \]
  so
  \[
    \left(\sfix{\eta}\right)_D = (\FIX (\Lambda \eta))_D = \FIX((\Lambda \eta)_D).
  \]
  This completes the proof.
  \qed
\end{proof}

\begin{proposition}
  \label{prop:main:16}
  {\CFP} is a cartesian closed 2-category.
\end{proposition}

\begin{proof}
  The cartesian-closed and 2-categorical structures are inherited from \O.

  {\CFP} is exactly the category $\mb{Kind}$ of \cite[Definition~7.3.11]{fiore_1994:_axiom_domain_theor}: ``the 2-category of small $\mb{Cpo}$-categories with ep-zero and colimits of $\omega$-chains of embeddings, $\mb{Cpo}$-functors and natural transformations''.
  A \defin{$\mb{Cpo}$-category}~\cite[25]{fiore_1994:_axiom_domain_theor} is exactly an \O-category.
  An \defin{ep-zero}~\cite[Definition~7.1.1]{fiore_1994:_axiom_domain_theor} is a zero object such that every morphism with it as a source is an embedding.
  The objects of {\CFP} are \O-categories that support canonical fixed points, \ie, they have strict morphisms and an initial object.
  By \cref{prop:main:14}, this implies they have an ep-zero.
  The objects of {\CFP} are \O-cocomplete \O-categories $\mb{K}$.
  This mean that every $\omega$-chain of embeddings in $\mb{K}$ has an \O-colimit.
  By \cref{prop:33}, this implies that every $\omega$-chain of embeddings has a colimit.
  Conversely, if every $\omega$-chain of embeddings in $\mb{K}$ has a colimit, then by \cref{prop:33} it has an \O-colimit, so is $\O$-cocomplete.
  So $\mb{Kind}$ and {\CFP} have the same objects.
  $\mb{Cpo}$-functors are exactly locally continuous functors~\cite[25]{fiore_1994:_axiom_domain_theor}.
  So $\mb{Kind}$ and {\CFP} are exactly the same category.
  It is a cartesian closed 2-category by \cite[Theorem~7.3.11]{fiore_1994:_axiom_domain_theor}.\qed
\end{proof}

Recall that we write $F_D$ for the partial application $\Lambda F D = F(D,{-})$.

\begin{proposition}[\Cref{prop:32}]
  Let $\mb{E}$ and $\mb{D}$ be $\mb{O}$-categories.
  Assume $\mb{E}$ supports canonical fixed points.
  Let $F : \mb{D} \times \mb{E} \to \mb{E}$ be a locally continuous functor.
  There exist natural transformations
  \begin{align*}
    \ms{Unfold}^F &: \sfix F \nto F \circ \langle \ms{id}_{\mb{D}}, \sfix F \rangle\\
    \ms{Fold}^F &: F \circ \langle \ms{id}_{\mb{D}}, \sfix F \rangle \nto \sfix F
  \end{align*}
  that form a natural isomorphism $\sfix F \cong F \circ \langle \ms{id}_{\mb{D}}, \sfix F \rangle$.
  Let $D$ be an object of $\mb{D}$.
  The $D$-components for these natural transformations are given by
  \begin{align*}
    \ms{Unfold}^F_D &= \ms{unfold}_{\lk{\bot}{\bot}{F_D}} : \GFIX\lk{\bot}{\bot}{F_D} \to \UNF\lk{\bot}{\bot}{F_D}\\
    \ms{Fold}^F_D &= \ms{fold}_{\lk{\bot}{\bot}{F_D}} : \UNF\lk{\bot}{\bot}{F_D} \to \GFIX\lk{\bot}{\bot}{F_D}
  \end{align*}
  where $\ms{unfold}$ and $\ms{fold}$ are the natural isomorphisms given by \cref{prop:main:3}.
  The definitions of $\ms{Fold}^F$ and $\ms{Unfold}^F$ are natural in $F$.
  Given any natural transformation $\eta : F \nto G$, the following two squares commute:
  \[
    \begin{tikzcd}[column sep=4em, arrows=Rightarrow]
      \sfix{F}
      \ar[r, "{\ms{Unfold}^F}"]
      \ar[d, swap, "{\sfix\eta}"]
      & F\circ\langle\ms{id}, \sfix{F} \rangle
      \ar[d, "{\eta \ast \langle \ms{id}, \sfix{\eta} \rangle}"]
      &F\circ\langle\ms{id}, \sfix{F} \rangle
      \ar[r, "{\ms{Fold}^F}"]
      \ar[d, swap, "{\eta \ast \langle \ms{id}, \sfix{\eta} \rangle}"]
      & \sfix{F}
      \ar[d, "{\sfix\eta}"]\\
      \sfix{G} \ar[r, "{\ms{Unfold}^G}"]
      & G\circ\langle\ms{id}, \sfix{G} \rangle
      & G\circ\langle\ms{id}, \sfix{G} \rangle
      \ar[r, "{\ms{Fold}^G}"]
      & \sfix{G}
    \end{tikzcd}
  \]
\end{proposition}

\begin{proof}[of \cref{prop:32}]
  We begin by noting that the components of $\ms{Unfold}$ and $\ms{Fold}$ have the right domains and codomains.
  We first note that by \cref{prop:main:4} and the definition of $\FIX$,
  \[
    \GFIX\lk{\bot}{\bot}{F_D} = \FIX(F_D) = \sfix{F}D.
  \]
  Also,
  \[
    \UNF\lk{\bot}{\bot}{F_D} = F_D(\FIX(F_D)) = F(\sfix{F}D, D) = \left(F \circ \langle \ms{id}_{\mb{D}}, \sfix{F} \rangle \right)D.
  \]
  Because $\ms{Unfold}$ and $\ms{Fold}$ are component-wise isomorphisms, we need only show that the components form a natural family.

  Let $f : C \to D$ be an arbitrary morphism in $\mb{D}$.
  Then $F_f = \Lambda Ff : F_C \nto F_D$ is a natural transformation.
  To that $\ms{Unfold}$ is natural, must show that the following square commutes:
  \[
    \begin{tikzcd}[column sep=5em]
      \sfix{F}C
      \ar[r, "{\ms{Unfold}^F_C}"]
      \ar[d, swap, "{\sfix{F}f}"]
      &
      (\sfix{F} \circ \langle \ms{id}_{\mb{D}}, \sfix F \rangle)C
      \ar[d, "{(\sfix{F} \circ \langle \ms{id}_{\mb{D}} \sfix F \rangle)f}"]
      \\
      \sfix{F}D
      \ar[r, "{\ms{Unfold}^F_D}"]
      &
      (\sfix{F} \circ \langle \ms{id}_{\mb{D}}, \sfix F \rangle) D.
    \end{tikzcd}
  \]
  This is exactly the following square, which commutes by \cref{prop:main:3}:
  \[
    \begin{tikzcd}[column sep=7em]
      \GFIX\lk{\bot}{\bot}{F_C}
      \ar[r, "{\ms{unfold}_{\lk{\bot}{\bot}{F_C}}}"]
      \ar[d, swap, "{\GFIX(\ms{id}, F_f)}"]
      &
      \UNF\lk{\bot}{\bot}{F_C}
      \ar[d, "{\UNF\lk{\bot}{\bot}{F_f}}"]
      \\
      \GFIX\lk{\bot}{\bot}{F_D}
      \ar[r, "{\ms{unfold}_{\lk{\bot}{\bot}{F_D}}}"]
      &
      \UNF\lk{\bot}{\bot}{F_D}.
    \end{tikzcd}
  \]
  Because $\ms{Fold}$ and $\ms{Unfold}$ form an isomorphism, we conclude by~\cite[Lemma~1.5.10]{riehl_2016:_categ_theor_contex} that $\ms{Fold}$ is also natural.

  Next, we show that the definitions of $\ms{Fold}^F$ and $\ms{Unfold}^F$ are natural in $F$, \ie, that for any natural transformation $\eta : F \nto G$, the following two squares commute:
  \[
    \begin{tikzcd}[column sep=4em, arrows=Rightarrow]
      \sfix{F}
      \ar[r, "{\ms{Unfold}^F}"]
      \ar[d, swap, "{\sfix\eta}"]
      & F\circ\langle\ms{id}, \sfix{F} \rangle
      \ar[d, "{\eta \ast \langle \ms{id}, \sfix{\eta} \rangle}"]
      &F\circ\langle\ms{id}, \sfix{F} \rangle
      \ar[r, "{\ms{Fold}^F}"]
      \ar[d, swap, "{\eta \ast \langle \ms{id}, \sfix{\eta} \rangle}"]
      & \sfix{F}
      \ar[d, "{\sfix\eta}"]\\
      \sfix{G} \ar[r, "{\ms{Unfold}^G}"]
      & G\circ\langle\ms{id}, \sfix{G} \rangle
      & G\circ\langle\ms{id}, \sfix{G} \rangle
      \ar[r, "{\ms{Fold}^G}"]
      & \sfix{G}
    \end{tikzcd}
  \]
  Because $\ms{Fold}^F$ and $\ms{Unfold}^F$ are isomorphisms, one square commutes if and only if the other does.
  The left square commutes if and only if every component does, \ie, if and only if for every object $D$ of $\mb{D}$, the following square commutes:
  \[
    \begin{tikzcd}[column sep=4em]
      \sfix{F}D
      \ar[r, "{\ms{Unfold}^F_D}"]
      \ar[d, swap, "{(\sfix\eta)_D}"]
      & (F\circ\langle\ms{id}, \sfix{F} \rangle)D
      \ar[d, "{(\eta \ast \langle \ms{id}, \sfix{\eta} \rangle)_D}"]
      \\
      \sfix{G}D \ar[r, "{\ms{Unfold}^G_D}"]
      & (G\circ\langle\ms{id}, \sfix{G} \rangle)D.
    \end{tikzcd}
  \]
  By \cref{prop:main:4} and the above, it is exactly the following square:
  \[
    \begin{tikzcd}[column sep=7em]
      \GFIX\lk{\bot}{\bot}{F_D}
      \ar[r, "{\ms{unfold}_{\lk{\bot}{\bot}{F_D}}}"]
      \ar[d, swap, "{\GFIX(\ms{id}, (\Lambda\eta)_D)}"]
      &
      \UNF\lk{\bot}{\bot}{F_D}
      \ar[d, "{\UNF(\ms{id}, (\Lambda\eta)_D)}"]
      \\
      \GFIX\lk{\bot}{\bot}{G_D}
      \ar[r, "{\ms{unfold}_{\lk{\bot}{\bot}{G_D}}}"]
      &
      \UNF\lk{\bot}{\bot}{G_D}.
    \end{tikzcd}
  \]
  It commutes by \cref{prop:main:3}.
  We conclude that the definition of $\ms{Unfold}^F$ (and also $\ms{Fold}^F$) is natural in $F$.
  \qed
\end{proof}

\propMainBD*

\begin{proof}
  To show \cref{eq:main:10}, we first observe that by naturality of $\Lambda$,
  \[
    \Lambda(F_G) = \Lambda(F \circ (G \times \ms{id}_{\mb{E}})) = (\Lambda F) \circ G.
  \]
  Then by \cref{prop:main:4},
  \[
    \sfix{F_G} = ([ \ms{id}_{\mb{D}} \to \FIX ] \circ \Lambda)(F_G) = \left(\left([ \ms{id}_{\mb{D}} \to \FIX ] \circ \Lambda\right)(F)\right) \circ G = \sfix{F} \circ G.
  \]
  \Cref{eq:main:11} follows from \cref{eq:main:10} and the following calculation:
  \begin{align*}
    &F_G \circ \langle \ms{id}_{\mb{C}}, \sfix{F_G} \rangle\\
    &= F \circ (G \times \ms{id}_{\mb{E}}) \circ \langle \ms{id}_{\mb{C}}, \sfix{F} \circ G \rangle\\
    &= F \circ \langle G \circ \ms{id}_{\mb{C}}, \ms{id}_{\mb{E}} \circ \sfix{F} \circ G \rangle\\
    &= F \circ \langle \ms{id}_{\mb{D}} \circ G, \sfix{F} \circ G \rangle\\
    &= F \circ \langle \ms{id}_{\mb{D}}, \sfix{F} \rangle \circ G.
  \end{align*}
  We use naturality and the 2-categorical structure of {\O} to show \cref{eq:main:14}:
  \begin{align*}
    &\sfix{\left(\phi \ast (\gamma \times \ms{id})\right)}\\
    &= ([ \ms{id}_{\mb{D}} \to \FIX ] \circ \Lambda)\left(\phi \ast (\gamma \times \ms{id})\right)\\
    &= \left(([ \ms{id}_{\mb{D}} \to \FIX ] \circ \Lambda)\phi\right) \ast \gamma\\
    &= \sfix{\phi} \ast \gamma.
  \end{align*}
  \Cref{eq:main:12} follows from \cref{prop:32}.
  Let $C$ be an arbitrary object in $\mb{C}$, then the $C$-component of $\ms{Fold}^{F_G}$ is:
  \[
    \ms{Fold}^{F_G}_C = \ms{fold}_{\bot,\bot,F^G_C} = \ms{fold}_{\bot,\bot,F_{GC}} = \ms{Fold}^F_{GC} = (\ms{Fold}^FG)_C.
  \]
  The proof of \cref{eq:main:13} is analogous.
  This completes the proof.\qed
\end{proof}

\section{Proofs for \cref{sec:canon-fixed-points}}
\label{sec:proofs-canon-fixed-points}

\begin{proposition}
  \label{prop:main:20}
  Let $\mb{K}$ be an $\mb{O}$-category with an initial object and strict morphisms.
  Let $F : \mb{K} \to \mb{K}$ be locally continuous.
  The following defines a functor $\Cone^F : \algc{\mb{K}}{F} \to \cel{\Cone_{\mb{K}}\left(\lk{\bot}{\bot}{F}, {-}\right)}$:
  \begin{itemize}
  \item on objects: $\Cone^F(A, a) = (\alpha, A)$ where $\alpha : \Omega\lk{\bot}{\bot}{F} \nto A$ is inductively defined by $\alpha_0 = \bot_A$ and $\alpha_{n + 1} = a \circ F\alpha_n$
  \item on morphisms: $\Cone^Ff = f$.
  \end{itemize}
  ${\Cone^F}$ restricts to a functor $\algc{(\mb{K}^e)}{F} \to \cel{\Cone_{\mb{K}^e}\left(\Omega\lk{\bot}{\bot}{F}, {-}\right)}$ whenever $F$ is locally continuous.
  ${\Cone^F}$ and its restriction are locally continuous.
\end{proposition}

\begin{proof}
  We begin by checking that the functor is well-defined on objects.
  Let $(A, a)$ be an $F$-algebra.
  We show that $\alpha$ is a cocone on $\Omega\lk{\bot}{\bot}{F}$ with nadir $A$.
  We must show that for all $n$, $\alpha_n = \alpha_{n + 1} \circ \Omega\lk{\bot}{\bot}{F}(n \to n + 1)$.
  We do so by induction on $n$.
  When $n = 0$, we have by initiality that
  \[
    \alpha_0 = \bot_A = a \circ F\bot_A \circ \bot_{F_\bot} = \alpha_1 \circ \Omega\lk{\bot}{\bot}{F}(0 \to 1).
  \]
  Assume the result for some $n$, then
  \begin{align*}
    \alpha_{n+1} &= a \circ F\left(\alpha_{n}\right)\\
                 &= a \circ F\left(\alpha_{n+1} \circ \Omega\lk{\bot}{\bot}{F}(n \to n + 1)\right)\\
                 &= a \circ F(\alpha_{n+1}) \circ F\left(\Omega\lk{\bot}{\bot}{F}(n \to n + 1)\right)\\
                 &= a \circ F(\alpha_{n+1}) \circ \Omega\lk{\bot}{\bot}{F}(n + 1 \to n + 2)\\
                 &= \alpha_{n + 2} \circ \Omega\lk{\bot}{\bot}{F}(n + 1 \to n + 2).
  \end{align*}
  We conclude that $\alpha$ is a cocone.
  By \cref{prop:main:14}, $\bot_A$ is an embedding.
  Recall that locally continuous functors preserve embeddings.
  If $a$ is also an embedding, then induction gives that each $\alpha_n$ is an embedding, \ie, that $(\alpha, A)$ lies in $\cel{\Cone_{\mb{K}^e}\left(\Omega\lk{\bot}{\bot}{F}, {-}\right)}$.

  The action of $\Cone^F$ on morphisms is clearly functorial and locally continuous.
  Let $(A, a)$ and $(B, b)$ be $F$-algebras and let $(\alpha, A)$ and $(\beta, B)$ be their respective images under $\Cone^F$.
  We must show that if $f : (A, a) \to (B, b)$ is an $F$-algebra homomorphism, then it is a morphism of cocones.
  In particular, we must show that for all $n \in \N$, $f \circ \alpha_n = \beta_n$.
  We do so by induction on $n$.
  When $n = 0$, we have by initiality that $f \circ \alpha_0 = \bot_B = \beta_0$.
  Assume the result for some $n$.
  Because $f$ is an $F$-algebra homomorphism, $f \circ a = b \circ Ff$.
  It follows that:
  \[
    f \circ \alpha_{n + 1} = f \circ a \circ F\alpha_n = b \circ Ff \circ F\alpha_n = b \circ F(f \circ \alpha_n) = b \circ F\beta_n = \beta_{n + 1}.
  \]
  We conclude the result by induction.
  \qed
\end{proof}

We recognize the cocone given in \cref{item:3} of \cref{prop:41} as cocone $\Cone^F(A, a)$ given by \cref{prop:main:20}.

\propEB*

\begin{proof}
  We begin by showing \cref{item:16}.
  We show that $(\FIX(F), \ms{fold})$ is initial.
  Let $(A, a)$ be any other $F$-algebra, and let $\alpha : \Omega\lk{\bot}{\bot}{F} \nto A$ be given as in the statement.
  It is the cocone $\Cone^F(A, a) = (\alpha, A)$ given by \cref{prop:main:20}.

  By \cref{prop:33}, the mediating morphism $(\kappa, \FIX(F)) \to (\alpha, A)$ of cocones is the embedding
  \[
    \phi = \dirsup_{n \in \N} \alpha_n \circ \kappa_n^p.
  \]
  We claim that it is an $F$-algebra homomorphism $(\FIX(F), \ms{fold}) \to (A, a)$.
  We use the fact that $F$ is locally continuous and compute:
  \begin{align*}
    &\phi \circ \ms{fold}\\
    &= \left(\dirsup_{n \in \N} \alpha_n \circ \kappa_n^p\right) \circ \left(\dirsup_{n \in \N} \kappa_{n+1} \circ F\kappa_n^p\right)\\
    &= \dirsup_{n \in \N} \alpha_{n + 1} \circ \kappa_{n + 1}^p \circ \kappa_{n+1} \circ F\kappa_n^p\\
    &= \dirsup_{n \in \N} \alpha_{n + 1} \circ \circ F\kappa_n^p\\
    &= \dirsup_{n \in \N} a \circ F\alpha_n \circ \circ F\kappa_n^p\\
    &= a \circ F\left(\dirsup_{n \in \N} \alpha_n \circ \kappa_n^p\right)\\
    &= a \circ F\phi.
  \end{align*}
  So $\phi$ is an $F$-algebra homomorphism $(\FIX(F), \ms{fold}) \to (A, a)$.

  We now show that $\phi$ is the unique such morphism.
  Let $\gamma$ be any other $F$-algebra homomorphism $(\FIX(F), \ms{fold}) \to (A, a)$.
  We show that it is a morphism of $(\kappa, \FIX(F)) \to (\alpha, A)$ of cocones.
  We use the fact that $(\kappa, \FIX(F))$ is the initial cocone to conclude that $\gamma = \phi$.
  We must show that for all $n \in \N$, $\alpha_n = \gamma \circ \kappa_n$.
  The case $n = 0$ is immediate by initiality: $\alpha_0$ and $\kappa_0$ are both morphisms $\bot \to A$.
  Assume the result for some $n$.
  Recognize $\ms{fold}$ as the mediating morphism from the cocone $(F\kappa, F(\FIX(F))) \to ((\kappa_{n+1})_n, \FIX F)$ in $\cel{\Cone(F\Omega\lk{\bot}{\bot}{F}, {-})}$.
  So $\kappa_{n + 1} = \ms{fold} \circ F\kappa_n$ for all $n \in \N$.
  We use the fact that $\gamma \circ \ms{fold} = a \circ F\gamma$ and compute:
  \begin{align*}
    &\gamma \circ \kappa_{n + 1}\\
    &= \gamma \circ \ms{fold} \circ F\kappa_n\\
    &= a \circ F\gamma \circ F\kappa_n\\
    &= a \circ F(\gamma \circ \kappa_n)\\
    &= a \circ F(\alpha_n)\\
    &= \alpha_{n + 1}.
  \end{align*}
  This establishes the first result.

  The proof of \cref{item:17} is dual to the proof of \cref{item:16}.

  We now show \cref{item:18}.
  Let $(\Gamma, \gamma)$ be an arbitrary $F$-algebra where $\gamma$ is an embedding.
  Then $(\Gamma, \gamma^p)$ is clearly an $F$-coalgebra.
  Let $\alpha : \Omega\lk{\bot}{\bot}{F} \nto C$ be given by $\alpha_0 = \bot_C$ and $\alpha_{n + 1} = c \circ F\alpha_n$.
  Let $\beta : C \nto \Omega\lk{\bot}{\bot}{F}$ be given by $\beta_{n + 1} = F\beta_n \circ c^p$.
  By \cref{item:16,item:17},
  \[
    \phi = \dirsup_{n \in \N} \alpha_n \circ \kappa_n^p \quad\text{and}\quad\rho = \dirsup_{n \in \N} \kappa_n \circ \beta_n.
  \]
  We claim that $\beta \circ \alpha = \ms{id}_{\Omega\lk{\bot}{\bot}{F}}$.
  We proceed by induction on $n$ to show that for all $n \in \N$, $\beta_n \circ \alpha_n = \ms{id}_{F^n\bot}$.
  The case $n = 0$ is obvious:
  \[
    \beta_0 \circ \alpha_0 = \bot_C^p \circ \bot_C = \bot_\bot = \ms{id}_{\bot} = \ms{id}_{F^0\bot}.
  \]
  Assume the result for some $n$, then:
  \begin{align*}
    &\beta_{n + 1} \circ \alpha_{n + 1}\\
    &= F\beta_n \circ c^p \circ c \circ F\alpha_n\\
    &= F\beta_n \circ F\alpha_n\\
    &= F(\beta_n \circ \alpha_n)\\
    &= F(\ms{id}_{F^n\bot})\\
    &= \ms{id}_{F^{n  + 1}\bot}.
  \end{align*}
  We also have that $\alpha \circ \beta \sqsubseteq \ms{id}_C$.
  We proceed by induction to show that for all $n \in \N$, $\alpha_n \circ \beta_n \sqsubseteq \ms{id}_C$.
  The case $n = 0$ is immediate by definition of e-p-pair.
  Assume the result for some $n$, then
  \begin{align*}
    &\alpha_{n + 1} \circ \beta_{n + 1}\\
    &= c \circ F\alpha_n \circ F\beta_n \circ c^p\\
    &= c \circ F(\alpha_n \circ \beta_n) \circ c^p\\
    &\sqsubseteq c \circ F(\ms{id}_C) \circ c^p
    &= c \circ c^p\\
    &\sqsubseteq \ms{id}_C.
  \end{align*}
  We now compute
  \begin{align*}
    &\rho \circ \phi\\
    &= \left(\dirsup_{n \in \N} \kappa_n \circ \beta_n\right) \circ \left(\dirsup_{n \in \N} \alpha_n \circ \kappa_n^p\right)\\
    &= \dirsup_{n \in \N} \kappa_n \circ \beta_n \circ \alpha_n \circ \kappa_n^p\\
    &= \dirsup_{n \in \N} \kappa_n \circ \kappa_n^p\\
    &= \ms{id}_C,
  \end{align*}
  where the last equality is because $\kappa$ is an \O-colimit.
  Similarly,
  \begin{align*}
    &\rho \circ \phi\\
    &= \left(\dirsup_{n \in \N} \alpha_n \circ \kappa_n^p\right) \circ \left(\dirsup_{n \in \N} \kappa_n \circ \beta_n\right)\\
    &= \dirsup_{n \in \N} \alpha_n \circ \kappa_n^p \circ \kappa_n \circ \beta_n\\
    &= \dirsup_{n \in \N} \alpha_n \circ \beta_n\\
    &\sqsubseteq \dirsup_{n \in \N} \ms{id}_C\\
    &= \ms{id}_C
  \end{align*}
  So we conclude that $(\phi, \rho)$ form an e-p-pair.\qed
\end{proof}

\propEF*

\begin{proof}
  We begin by showing \cref{item:3}.
  Let $(G, \gamma)$ be an arbitrary $F$-algebra.
  For every object $D$, $(\sfix{F}D, \ms{Fold}_D)$ is the initial $F_D$-algebra by \cref{prop:main:4,prop:32,prop:41}.
  This implies there exists a unique $F_D$-homomorphism $\phi_D : \sfix{F}D \to GD$ making the following square commute:
  \[
    \begin{tikzcd}
      F_D(\sfix{F}D)
      \ar[r, "{\ms{Fold}_D}"]
      \ar[d, swap, "{F_D\phi_D}"]
      &
      \sfix{F}D
      \ar[d, "{\phi_D}"]\\
      F_D(GD)
      \ar[r, "{\gamma_D}"]
      & GD.
    \end{tikzcd}
  \]
  We claim that these morphisms $\phi_D$ assemble into a natural transformation $\phi : \sfix{F} \to G$.
  It will immediately follow that $\phi$ is an $F$-algebra homomorphism from $(\sfix{F}, \ms{Fold})$ to $(G, \gamma)$.

  To show that $\phi$ is natural, let $f : A \to B$ be an arbitrary morphism in $\mb{D}$.
  We must show that the following square commutes:
  \[
    \begin{tikzcd}
      \sfix{F}A
      \ar[r, "{\phi_A}"]
      \ar[d, swap, "{\sfix{F}f}"]
      &
      GA
      \ar[d, "{Gf}"]
      \\
      \sfix{F}{B}
      \ar[r, "{\phi_B}"]
      & GB.
    \end{tikzcd}
  \]
  Given an $L : \mb{E} \to \mb{E}$, write $L^\omega$ for the functor $\Omega\lk{\bot}{\bot}{L}$.
  Let $\alpha : F_A^\omega \nto \sfix{F}A$ and $\beta : F_B^\omega \to \sfix{F}B$ be the chosen \O-colimits.
  Because $\mb{E}$ supports canonical fixed points, these cocones are colimiting in both $\mb{E}^e$ and $\mb{E}$.
  Let $\nu^A : F_A^\omega \to GA$ and $\nu^B : F_B^\omega \to GB$ respectively be the cocones the $F_A$-algebra $(GA, \gamma_A)$ and $F_B$-algebra $(GB, \gamma_B)$ induce via \cref{prop:main:20}.
  By comparing \cref{prop:main:20,prop:41}, we observe that $\phi_A : (\alpha, \sfix{F}A) \to (\nu^A, GA)$ and $\phi_B : (\beta, \sfix{F}B) \to (\nu^B, GB)$ are cocone morphisms.
  Write $F_f$ for the natural transformation $\Lambda F f : F_A \nto F_B$.
  We then have the following diagram in $\mb{E}$, where cocones and natural transformations are indicated by $\Rightarrow$.
  \begin{equation}
    \label[diagram]{eq:283}
    \begin{tikzcd}
      \sfix{F}A
      \ar[rr, "{\phi_A}"]
      \ar[ddd, swap, "{\sfix{F}f}"]
      &
      &
      GA \ar[ddd, "{Gf}"]
      \\
      &
      F_A^\omega
      \ar[d, Rightarrow, "{F_f^\omega}"]
      \ar[ul, Rightarrow, "{\alpha}"]
      \ar[ur, Rightarrow, swap, "{\nu^A}"]
      &
      \\
      &
      F_B^\omega
      \ar[dl, swap, Rightarrow, "{\beta}"]
      \ar[dr, Rightarrow, "{\nu^B}"]
      &
      \\
      \sfix{F}B
      \ar[rr, "{\phi_B}"]
      &
      &
      GB.
    \end{tikzcd}
  \end{equation}
  We show that $\phi_B \circ \sfix{F}f$ and $Gf \circ \phi_A$ are both mediating morphisms from the colimiting cone $\alpha$ to the cocone $\nu^B \circ F^\omega_f$.
  It will then follow by uniqueness of mediating morphisms that they are equal and that $\phi$ is natural.

  We begin with $\phi_B \circ \sfix{F}f$.
  By definition of $\sfix{F}f$, $\sfix{F}f$ is a mediating morphism from $\alpha$ to $\beta \circ F_f^\omega$.
  By the remarks above, $\phi_B$ is a mediating morphism from $\beta$ to $\nu^B$, so it is also a mediating morphism from $\beta \circ F_f^\omega$ to $\nu^B \circ F_f^\omega$.
  So going around the left and bottom sides of \cref{eq:283}, we get a mediating morphism $\phi_B \circ \sfix{F}f$ from $\alpha$ to $\nu^B \circ F_f^\omega$.

  We next show that $Gf \circ \phi_A$ is a mediating morphism.
  By definition, $\phi_A$ is a mediating morphism from $\alpha$ to $\nu^A$.
  We must now show that $Gf$ is a mediating morphism from $\nu^A$ to $\nu^B \circ F_f^\omega$, \ie, we must show that for all $n$,
  \begin{equation}
    \label{eq:25}
    Gf \circ \nu^A_n = \nu^B_n \circ \left(F^\omega_f\right)_n : F_A^n\bot{\mb{D}} \to GB
  \end{equation}
  are equal morphisms.
  We do so by induction on $n$.
  When $n = 0$, initiality gives us
  \[
    Gf \circ \nu^A_0 = \bot_{GB} = \nu^B_0 \circ \left(F^\omega_f\right)_0.
  \]
  Assume the result for some $n$.
  To show the result for $n + 1$ we must show that the outer rectangle of \cref{eq:23} commutes:
  \begin{equation}
    \label[diagram]{eq:23}
    \begin{tikzcd}[column sep=3em]
      GA \ar[rrr, "{Gf}"] & & & GB
      \\
      &
      F_AGA
      \ar[r, "{F(f,Gf)}"]
      \ar[ul, swap, "{\gamma_A}"]
      &
      F_BGB
      \ar[ur, "{\gamma_B}"]
      &
      \\
      F_A^{n+1}\bot_{\mb{E}}
      \ar[rrr, swap, "{\left(F^\omega_f\right)_{n+1}}"]
      \ar[uu, "{\nu^A_{n+1}}"]
      \ar[ur, swap, "{F_A(\nu^A_n)}"]
      & & &
      F_B^{n+1}\bot_{\mb{E}}
      \ar[uu, swap, "{\nu^B_{n+1}}"]
      \ar[ul, "{F_B(\nu^B_n)}"]
    \end{tikzcd}
  \end{equation}
  The upper trapezoid commutes by definition of $F$-algebra and the assumption that $(G, \gamma)$ was an $F$-algebra.
  The two triangles of \cref{eq:23} commute by definition of $\nu^A_{n+1}$ and $\nu^B_{n+1}$ (\cf~\cref{prop:main:20}).
  The bottom trapezoid is equal to the perimeter of \cref{eq:24}:
  \begin{equation}
    \label[diagram]{eq:24}
    \begin{gathered}
      \xymatrix@C=5em{
        F_AGA \ar[r]^{F_A(Gf)} & F_AGB \ar[r]^{\left(F_f\right)_{GB}} & F_BGB\\
        F^{n+1}_A\bot_{\mb{E}} \ar[u]^{F_A(\nu^A_n)} \ar[r]_{F_A\left(\left(F_f^\omega\right)_{n}\right)} & F_AF_B^n\bot_{\mb{E}} \ar[r]_{\left(F_f\right)_{F_B^n\bot_{\mb{E}}}} \ar[u]^{F_A(\nu^B_n)} & F^{n+1}_B\bot_{\mb{E}} \ar[u]_{F_B(\nu^B_n)}.
      }
    \end{gathered}
  \end{equation}
  Indeed, the top morphism of \cref{eq:24} is exactly $F(f,Gf)$:
  \[
    F(f,Gf) = F(f, \ms{id}_{GB}) \circ F(\ms{id}_A, Gf) = \left(F_f\right)_{GB} \circ F_A(Gf).
  \]
  The bottom morphisms are equal by definition of $F_f^\omega$ (\cref{eq:main:3}):
  \begin{align*}
    &\left(F^\omega_f\right)_{n+1}\\
    &= \Omega(\ms{id}_\bot, F^\omega_f)_{n + 1}\\
    &= F_f^{(n+1)} \ast \ms{id}_{\bot_{\mb{E}}}\\
    &= \left(F_f^{(n+1)}\right)_{\bot_{\mb{E}}}\\
    &= \left(F_f \ast F_f^{(n)}\right)_{\bot_{\mb{E}}}\\
    &= \left(F_f\right)_{F_B^n\bot_{\mb{E}}} \circ F_A\left(\left(F_f^{(n)}\right)_{\bot_{\mb{E}}}\right)\\
    &=  \left(F_f\right)_{F_B^n\bot_{\mb{E}}} \circ F_A\left(\left(F_f^\omega\right)_n\right).
  \end{align*}
  To see that \cref{eq:24} commutes, we note that the left square commutes by applying $F_A$ to the square given by the induction hypothesis.
  The right square commutes by naturality of $F_f$.
  By pasting, the perimeter commutes.
  So we conclude that the bottom trapezoid of \cref{eq:23} commutes.

  By pasting the two trapezoids and two triangles, we get that \cref{eq:23} commutes.
  \Cref{eq:25} then holds by induction, so $Gf$ is a mediating morphism from $\nu^A$ to $\nu^B \circ F_f^\omega$.
  By composing around the top and right sides of \cref{eq:283}, we get a mediating morphism from $\alpha$ to $\nu^B \circ F_f^\omega$.

  By the remarks following \cref{eq:283}, we conclude that $\phi$ is a natural transformation from $\sfix{F}$ to $G$.
  Because every component was an embedding, it is a natural family of embeddings.

  Dualizing the above argument gives us that $(\sfix{F}, \ms{Unfold})$ is a terminal $F$-coalgebra.
  The proof of \cref{item:4} is obtained from the proof of \cref{item:3} by replacing all occurrences of ``cocone'' by ``cone'', ``colimit'' by ``limit'', ``embedding'' by ``projection'', and ``initial'' by ``terminal''.

  To show \cref{item:7}, let $(A, \alpha)$ be an arbitrary $H$-algebra and assume $\alpha$ is a natural embedding.
  By definition, $\alpha^p$ is a natural projection, and $(A, \alpha^p)$ is then an $H$-coalgebra.
  Let $\phi$ and $\rho$ be the natural transformations given by \cref{item:3,item:4} and let $D$ be an object of $\mb{D}$.
  Then the $D$-components of $\phi$ and $\rho$ are respectively the mediating morphisms given by the initiality of the $F_D$-algebra $(\sfix{F}D, \ms{fold}_D)$ and the terminality of the the $F_D$-coalgebra $(\sfix{F}D, \ms{unfold}_D)$.
  By \cref{prop:41}, these form an e-p-pair.
  Because $D$ was arbitrary, it immediately follows that $(\phi, \rho)$ form an e-p-pair in $[\mb{D} \lcto \mb{E}]$.
  \qed
\end{proof}

\end{document}